\documentclass[11pt,reqno]{amsart}

\usepackage{floatrow}

\usepackage{amsmath}
\usepackage{amssymb}
\usepackage[left=1in,top=1in,right=1in,bottom=1in]{geometry}
\usepackage{epsfig}
\usepackage[normalem]{ulem}
\usepackage[usenames,dvipsnames]{color}
\usepackage{todonotes}
\usepackage[colorlinks, citecolor=black, linkcolor=black, urlcolor=black]{hyperref}
\usepackage{bm}
\usepackage{algorithm}
\usepackage{algpseudocode}

\newenvironment{varalgorithm}[1]
  {\algorithm\renewcommand{\thealgorithm}{#1}}
  {\endalgorithm}

\usepackage{enumitem}
\newcommand{\subscript}[2]{$#1 _ #2$}

\allowdisplaybreaks

\usepackage{hyperref}

\usepackage{tikz}
\usetikzlibrary{backgrounds}
\usetikzlibrary{intersections}

\newtheorem{theorem}{Theorem}[section]
\newtheorem{lemma}[theorem]{Lemma}
\newtheorem{proposition}[theorem]{Proposition}

\newcommand\eps{\varepsilon}
\newcommand{\E}{\mathbb E}

\newcommand{\Bin}{\mathrm{Bin}}

\newcommand{\Nn}{{\mathbb N}}

\newcommand{\scr}{\mathcal}
\newcommand{\mb}{\mathbb}

\newcommand{\ex}{\mathrm{ex}}

\theoremstyle{definition}
\newtheorem{remark}{Remark}

\def\ex{{\mathbb E}}

\def\W{{\mathcal W}}
\def\T{{\mathcal T}}
\title[Hamiltonian Cycles in the Semi-Random Graph Process]{A Fully Adaptive Strategy for Hamiltonian Cycles in the Semi-Random Graph Process}
\author{Pu Gao}
\address{Department of Combinatorics and Optimization, University of Waterloo, Waterloo, Canada}
\email{pu.gao@uwaterloo.ca}

\author{Calum MacRury}
\address{Department of Computer Science, University of Toronto, Toronto, Canada}
\email{cmacrury@cs.toronto.edu}

\author{Pawe\l{} Pra\l{}at}
\address{Department of Mathematics, Toronto Metropolitan University, Toronto, Canada}
\email{pralat@ryerson.ca}

\date{}

\begin{document}

\maketitle

\begin{abstract}
The semi-random graph process is a single player game in which the player is initially presented an empty graph on $n$ vertices. In each round, a vertex $u$ is presented to the player independently and uniformly at random. The player then adaptively selects a vertex $v$, and adds the edge $uv$ to the graph. For a fixed monotone graph property, the objective of the player is to force the graph to satisfy this property with high probability in as few rounds as possible.

We focus on the problem of constructing a Hamiltonian cycle in as few rounds as possible. In particular, we present an adaptive strategy for the player which achieves it in $\alpha n$ rounds, where $\alpha < 2.01678$ is derived from the solution to some system of differential equations. We also show that the player cannot achieve the desired property in less than $\beta n$ rounds, where $\beta > 1.26575$. 
These results improve the previously best known bounds and, as a result, the gap between the upper and lower bounds is decreased from 1.39162 to 0.75102. 
\end{abstract}


\section{Introduction and Main Results}

\subsection{Definitions} 
In this paper, we consider the \textbf{semi-random graph process} suggested by Peleg Michaeli, introduced formally in~\cite{beneliezer2019semirandom}, and studied recently in~\cite{beneliezer2020fast,gao2020hamilton,gao2021perfect,behague2022,Burova2022, macrury_2022} that can be viewed as a ``one player game''. The process starts from $G_0$, the empty graph on the vertex set $[n]:=\{1,\ldots,n\}$ where $n \ge 1$. In each \textbf{step} $t$, a vertex $u_t$ is chosen uniformly at random from $[n]$. Then, the player (who is aware of graph $G_t$ and vertex $u_t$) must select a vertex $v_t$ and add the edge $u_tv_t$ to $G_t$ to form $G_{t+1}$. The goal of the player is to build a (multi)graph satisfying a given property $\scr{P}$ as quickly as possible. It is convenient to refer to $u_t$ as a {\bf square}, and $v_t$ as a {\bf circle} so every edge in $G_t$ joins a square with a circle. We say that $v_t$ is paired to $u_t$ in step $t$. Moreover, we say that vertex $x \in [n]$ is \textbf{covered} by the square $u_t$ arriving at round $t$,
provided $u_t = x$. The analogous definition extends to the circle $v_t$. Equivalently, we may view $G_t$ as a directed graph where each arc  directs from $u_t$ to $v_t$, and thus we may use $(u_t,v_t)$ to denote the edge added in step $t$. For this paper, it is easier to consider squares and circles for counting arguments.

A \textbf{strategy} $\scr{S}$ is defined by specifying for each $n \ge 1$, a sequence of functions $(f_{t})_{t=1}^{\infty}$, where for each $t \in \mb{N}$, $f_t(u_1,v_1,\ldots, u_{t-1},v_{t-1},u_t)$ is a distribution over $[n]$
which depends on the vertex $u_t$, and the history of the process up until step $t-1$. Then, $v_t$ is chosen according to this distribution. If $f_t$ is an atomic distribution, then $v_t$ is determined by $u_1,v_1, \ldots ,u_{t-1},v_{t-1},u_t$. We then denote $(G_{i}^{\scr{S}}(n))_{i=0}^{t}$ as
the sequence of random (multi)graphs obtained by following the strategy $\scr{S}$ for $t$ rounds; where we shorten $G_{t}^{\scr{S}}(n)$
to $G_t$ or $G_{t}(n)$ when clear. 


Suppose $\scr{P}$ is a monotonely increasing property. Given a strategy $\scr{S}$ and a constant $0<q<1$, let $\tau_{\scr{P}}(\scr{S},q,n)$ be the minimum $t \ge 0$ for which $\mb{P}[G_{t} \in \scr{P}] \ge q$,
where $\tau_{\scr{P}}(\scr{S},q,n):= \infty$ if no such $t$ exists. Define
\[
\tau_{\scr{P}}(q,n) = \inf_{ \scr{S}} \tau_{\scr{P}}( \scr{S},q,n),
\]
where the infimum is over all strategies on $[n]$. 
Observe that for each $n \ge 1$, if $0 \le q_{1} \le q_{2} \le 1$, then $\tau_{\scr{P}}(q_1,n) \le \tau_{\scr{P}}(q_2,n) $ as $\scr{P}$ is increasing. Thus, the function $q\rightarrow \limsup_{n\to\infty} \tau_{\scr{P}}(q,n)$ is non-decreasing,
and so the limit,
\[
\tau_{\scr{P}}:=\lim_{q\to 1^-}\limsup_{n\to\infty} \frac{\tau_{\scr{P}}(q,n) }{n},
\]
is guaranteed to exist. The goal is typically to compute upper and lower bounds on $\tau_{\scr{P}}$
for various properties $\scr{P}$.

\subsection{Main Results} 

In this paper, we concentrate on the property of having a Hamiltonian cycle, which we denote by ${\tt HAM}$. As observed in~\cite{beneliezer2019semirandom}, if $G_t$ has a Hamiltonian cycle, then $G_t$ has minimum degree at least 2.
Thus, $\tau_{\tt HAM} \ge  \tau_{\scr{P}}  = \ln 2+ \ln(1+\ln2) \ge 1.21973$, where $\scr{P}$ corresponds to having minimum degree $2$. On the other hand, it is known that the famous $3$-out process is Hamiltonian with probability tending to 1 as $n \to \infty$ (\emph{a.a.s.})~\cite{bohman2009hamilton}. As the semi-random process can be coupled with the $3$-out process, we get that $\tau_{\tt HAM} \le 3$. A new upper bound was obtained in~\cite{gao2020hamilton} in terms of an optimal solution to an optimization problem whose value is believed to be $2.61135$ by numerical support. In the same paper, the lower bound mentioned above was shown to not be tight. The lower bound was increased by $\eps = 10^{-8}$ and so numerically negligible.

The upper bound on $\tau_{\tt HAM}$ of $3$ obtained by simulating the $3$-out process is \textbf{non-adaptive}. That is,
the strategy does \textit{not} depend on the history of the semi-random process. The above mentioned improvement proposed in~\cite{gao2020hamilton} uses an adaptive strategy but in a weak sense. The strategy consists of 4 phases, each lasting a linear number of rounds, and the strategy is adjusted \emph{only} at the end of each phase (for example, the player might identify vertices of low degree, and then focus on connecting
circles to them during the next phase). 

In this paper, we propose a fully adaptive strategy that pays attention to the graph $G_t$ and the position of $u_t$ for every single step $t$. As expected, such a strategy creates a Hamiltonian cycle substantially faster than our weakly adaptive strategy,
and it allows us to improve the upper bound from $2.61135$ to $2.01678$.

\begin{theorem}\label{thm:main_upper_bound}
$\tau_{\texttt{HAM}} \le \alpha \le 2.01678$, where $\alpha$ is derived
from a system of differential equations. 
\end{theorem}

The numerical results presented in this paper were obtained using the Julia language~\cite{Julia}. We would like to thank Bogumi\l{} Kami\'nski from SGH Warsaw School of Economics for helping us to implement it. The program is available on-line.\footnote{\texttt{https://math.ryerson.ca/\~{}pralat/research.html\#publications}} 

Moreover, by investigating some specific structures that are generated by the semi-random process, which guarantee the existence of a large set of families of edges that cannot simultaneously contribute to the construction of a Hamiltonian cycle, we improve the lower bound of $\ln 2+ \ln(1+\ln2) \ge 1.21973$ to $1.26575$. The structures we investigate in this work are different from the ones in~\cite{gao2020hamilton}. We attain a simpler proof than in~\cite{gao2020hamilton}, and a much stronger bound.

\begin{theorem}\label{thm:main_lower_bound}
Let $f(s)=2+e^{-3s}(s+1)\left(1-\frac{s^2}{2}-\frac{s^3}{3}-\frac{s^4}{8}\right)+e^{-2s}\left(2s+\frac{5s^2}{2}+\frac{s^3}{2}\right)-e^{-s}\left(3+2s\right)$, and let $\beta\approx 1.26575$ be the positive root of $f(s)-1=0$. Then,
$\tau_{\texttt{HAM}} \ge \beta$.
\end{theorem}

\subsection{Previous Results} 

Let us briefly describe a few known results on the semi-random process. In the very first paper~\cite{beneliezer2019semirandom}, it was shown that the process is general enough to approximate (using suitable strategies) several well-studied random graph models. In the same paper, the process was studied for various natural properties such as having minimum degree $k \in \Nn$ or having a fixed graph $H$ as a subgraph. In particular, it was shown that \emph{a.a.s.}\ one can construct $H$ in less than $n^{(d-1)/d} \omega$ rounds where $d \ge 2$ is the degeneracy of $G$ and $\omega = \omega(n)$ is any function that tends to infinity as $n \to \infty$. 
This property was recently revisited in~\cite{behague2022} where the conjecture from~\cite{beneliezer2019semirandom} was proven for any graph $H$: \emph{a.a.s.}\ it takes at least $n^{(d-1)/d} / \omega$ rounds to create $H$.

Another property that was studied in the context of semi-random processes is a property of having a perfect matching, which we denote by {\tt PM}. Since the $2$-out process has a perfect matching \emph{a.a.s.}~\cite{frieze1986maximum}, we immediately get that $\tau_{\texttt{PM}} \le 2$. By coupling the semi-random process with another random graph that is known to have a perfect matching \emph{a.a.s.}~\cite{pittel}, the bound can be improved to $1+2/e < 1.73576$. This bound was recently improved by the authors of this paper by investigating another fully adaptive algorithm~\cite{gao2021perfect}. The currently best upper bound is $\tau_{\texttt{PM}} < 1.20524$. In the same paper, the lower bound observed in~\cite{beneliezer2019semirandom} ($\tau_{\texttt{PM}} \ge \ln(2) > 0.69314$) was improved as well, and now we know that  $\tau_{\texttt{PM}} > 0.93261$~\cite{gao2021perfect}.

Finally, let us discuss what is known about the property of containing a given spanning graph $H$ as a subgraph. It was asked by Noga Alon whether for any bounded-degree $H$, one can construct a copy of $H$ \emph{a.a.s.}\ in $O(n)$ rounds.  This question was answered positively in a strong sense in~\cite{beneliezer2020fast}, in which it was shown that any graph with maximum degree $\Delta$ can be constructed \emph{a.a.s.}\ in $(3\Delta/2+o(\Delta))n$ rounds. They also proved that if $\Delta = \omega (\log(n))$, then this upper bound improves to $(\Delta/2 + o(\Delta))n$ rounds. Note that both of these upper bounds are asymptotic in $\Delta$. When $\Delta$ is constant in $n$, such as in both the perfect matching and Hamiltonian cycle setting, determining the optimal dependence on $\Delta$ for the number of rounds needed to construct $H$ remains open.

\section{Proof of Theorem \ref{thm:main_upper_bound}}

\subsection{Algorithmic Preliminaries} \label{sec:preliminaries}

In this section, we introduce some notation as well as the basic ideas used in the design of all of our strategies.

The main ingredient for proving Theorem~\ref{thm:main_upper_bound} is to specify a strategy which keeps augmenting or extending a path, until the path becomes Hamiltonian. Then, with a few more steps, the Hamiltonian path can be completed into a Hamiltonian cycle. 
Let us suppose that after $t \ge 0$ steps, we have constructed the graph $G_t$  which contains the path $P_t$. Define $U_t$ to be the set of vertices \textit{not} in $P_t$, 
which we refer to as the \textbf{unsaturated} vertices of $[n]$. It will be convenient to denote the (induced) distance between vertices $x, y \in V(P_t)$ on the path $P_t$ by $d_{P_t}(x,y)$.
We also define $d_{P_t}(x,Q) := \min_{q \in Q} d_{P_t}(x,q)$ for $x\in V(P_t)$ and $Q \subseteq V(P_t)$. 

Let us first assume that $u_{t+1}$ lands in $U_{t}$. In this case, we can clearly extend the path $P_{t}$ by an edge by choosing $v_{t+1}$ to be an endpoint of $P_{t}$. We call such a move a \textbf{(greedy) path extension}. Now, suppose that
$u_{t+1}$ lands on a vertex $x \in P_{t}$. In this case, we cannot perform a greedy path extension, however
we can still choose $v_{t+1}$ in a way that will help us extend the path in the future rounds. Specifically, 
set $v_{t+1}:=r$ for some $r \in U_{t}$, and \textbf{colour} the vertex $x$ as well as the edge $xr$. Suppose that in some round $i > t+1$,  $u_i$ lands on $y$ next
to the coloured vertex $x$ on $P_i$ (i.e., $d_{P_i}(x,y) =1$). In this case, set $v_{i} = r$. Observe
now that we can add $r$ to the current path by adding the edges $y r$ and $x r$ to it, and by removing
the edge $y x$. Thus, despite $u_s$ not landing on an unsaturated vertex, we are still able to 
perform a move which extends its length by one. We call such an operation a \textbf{path augmentation}.

\subsection{Proof Overview} \label{sec:upper_bound_overview}

In order to prove Theorem~\ref{thm:main_upper_bound}, we analyze a strategy which proceeds in
three distinct \textbf{stages}. In the first stage, we execute \ref{alg:degree_greedy}, an algorithm which makes greedy path extensions whenever possible, and otherwise sets up path augmentation operations for future rounds in a degree greedy manner. Specifically, $v_{t+1}$ is chosen amongst the unsaturated vertices of minimum coloured in-degree. This degree greedy decision is done to minimize the number of coloured vertices which are destroyed when path augmentations and extensions are made in later rounds. This stage lasts for $N$ \textbf{phases}, where $N$ is any non-negative integer that may be viewed as the parameter of the algorithm (here a phase is a contiguous set of steps shorter than the full stage). For the claimed (numerical) upper bound of Theorem~\ref{thm:main_upper_bound}, $N$ is set to $100$. Setting smaller values of the parameter $N$---in particular, setting $N=0$---yields an algorithm that is easier to analyse. Setting $N > 100$ can slightly improve the bound in Theorem~\ref{thm:main_upper_bound}, but the gain is rather insignificant. The second stage starts at some random step $t_0$ (i.e.\ $t_0-1$ is the total number of steps in stage one), and we execute \ref{alg:fully_randomized}, an algorithm which makes greedy path extensions whenever possible, and otherwise chooses $v_{t+1}$ randomly amongst the unsaturated vertices. We execute \ref{alg:fully_randomized} until we are left with $\eps n$ unsaturated vertices, where $\eps=\eps(n)$ tends to $0$ as $n \rightarrow \infty$ arbitrarily slowly. At this point, we proceed to the final stage where a clean-up algorithm is run, which also uses
path augmentations. Using elementary concentration inequalities we prove that
a Hamiltonian cycle can be  constructed in an additional $O(\sqrt{\eps}n)=o(n)$ steps.

In Section~\ref{warm-up-upper-bound}, we first describe \ref{alg:fully_randomized},
as it is easier to state and analyze than \ref{alg:degree_greedy}. Moreover, if we take $N=0$,
which corresponds to executing \ref{alg:fully_randomized} from the beginning, then we will be left with a path on all but $\eps n$ vertices after $\alpha^{*}n$ steps where $\alpha^{*} \le 2.07721$. Our third stage clean-up algorithm from Section \ref{sec:clean_up} allows us to complete the Hamiltonian cycle in another $o(n)$ steps. Thus, Sections \ref{warm-up-upper-bound} and \ref{sec:clean_up} provide a self-contained proof of an upper bound on $\tau_{\mathtt{HAM}}$ of $\alpha^{*} \le 2.07721$ (see Theorem~\ref{thm:warm_up}). Afterwards, in Section \ref{sec:degree_greedy} we formally state and analyze our first stage algorithm. This is the most technical section of the paper, as \ref{alg:degree_greedy} makes decisions in a more intelligent manner than \ref{alg:fully_randomized} which necessitates more random variables in its analysis. By executing these three stages in the aforementioned order, we attain the claimed upper bound of Theorem~\ref{thm:main_upper_bound}.

\subsection{A Fully Randomized Algorithm}\label{warm-up-upper-bound} 

In order to describe our algorithm,
it will be convenient to colour certain edges of $G_t$ red. This helps us define certain vertices used by our
algorithm for path augmentations. Specifically, $x \in V(P_t)$ is \textbf{one-red} provided
it is adjacent to precisely one red edge of $G_t$. Similarly, $x \in V(P_t)$ is \textbf{two-red}, provided it is adjacent to precisely two red edges of $G_t$.
We denote the one-red vertices and two-red vertices by $\scr{L}^{1}_t$
and $\scr{L}^{2}_t$, respectively, and refer to $\scr{L}_t := \scr{L}^{1}_t \cup \scr{L}^{2}_t$ as the \textbf{red} 
vertices of $G_t$.
By definition, $\scr{L}^{1}_t$ and $\scr{L}^{2}_t$ are disjoint.
It will also be convenient to maintain a set of \textbf{permissible vertices} $\scr{Q}_t \subseteq V(P_t)$
which specifies which uncoloured vertices on the path can be turned red. In order to simplify our analysis,
we specify the size of $\scr{Q}_t$ and ensure that it only contains vertices of path distance at least $3$ from the red vertices on $P_t$. 
Formally:
\begin{enumerate}
    \item[(i)] $|\scr{Q}_t| = |V(P_t)| - 5|\scr{L}_t|$. \label{eqn:permissible_size}
    \item[(ii)] If $\scr{L}_t \neq \emptyset$, then each $x \in \scr{Q}_t$ satisfies $d_{P_t}(x,\scr{L}_t) \ge 3$. \label{eqn:permissible_distance}
\end{enumerate}
When $\scr{L}_t = \emptyset$, we simply take $\scr{Q}_t = V(P_t)$.
Otherwise, since $|\{ x \in V(P_t) : d_{P_t}(x,\scr{L}_t) \le 2 \}| \le 5 |\scr{L}_t|$, we can maintain
these properties by initially taking $\{ x \in V(P_t) : d_{P_t}(x,\scr{L}_t) \ge 3\}$, and then (if needed) arbitrarily removing
$|\{ x \in V(P_t) : d_{P_t}(x,\scr{L}_t) \ge 3\}| - (|V(P_t)| - 5 |\scr{L}_t|)$ vertices from it.

Upon the arrival of $u_{t+1}$, there are four main cases our algorithm must handle. The first
two cases involve extending the length of the path, whereas the latter two describe what to do
when it is not possible to extend the path in the current round. 

\begin{enumerate}
    \item If $u_{t+1}$ lands within $U_t$, then greedily extend $P_{t}$. \label{eqn:greedy_extend}
    \item If $u_{t+1}$ lands at path distance one from some $x \in \scr{L}_t$, then augment $P_t$ via an arbitrary red edge of $x$. \label{eqn:path_augment}
    \item If $u_{t+1}$ lands in $\scr{Q}_t$, then choose $v_{t+1}$ u.a.r.
    amongst $U_t$, and colour $u_{t+1} v_{t+1}$ red. This case creates a one-red vertex. \label{eqn:one_cheery} 
    \item If $u_{t+1}$ lands in $\scr{L}^{1}_t$, then choose $v_{t+1}$ u.a.r. amongst $U_t$ and colour $u_{t+1} v_{t+1}$
     red. This case converts a one-red vertex to a two-red vertex. 
\end{enumerate}
In all the remaining cases, we choose $v_{t+1}$ arbitrarily, and interpret
the algorithm as \textit{passing} on the round, meaning the edge $u_tv_t$ will not be used to construct a Hamiltonian cycle. In particular, the algorithm passes rounds in which $u_{t+1}$ lands at path distance two from some $x \in \scr{L}_t$. This guarantees that no two red vertices are at distance two from each other and so when $u_{t+1}$ lands next to a red vertex, this neighbouring red vertex is uniquely identified.
Let us say that a red vertex is \textbf{well-spaced}, provided it is at distance
at least $3$ on the path from all other red vertices, and it is \textit{not} an endpoint of $P_t$.
Observe that each well-spaced red vertex yields precisely two vertices on $P_t$ where a path augmentation involving $u_{t+1}$ can occur. By construction, all but at most $2$ of the algorithm's red vertices are well-spaced.

We now formally describe step $t+1$  of the algorithm when $u_{t+1}$ is drawn u.a.r.\ from $[n]$.
Specifically, we describe how the algorithm chooses $v_{t+1}$, how it constructs $P_{t+1}$, and how it adjusts the colours of $G_{t+1}$, thus updating $\scr{L}^{1}_t$ and $\scr{L}^{2}_t$. 
\begin{varalgorithm}{$\mathtt{FullyRandomized}$}
\caption{Step $t+1$} 
\label{alg:fully_randomized}
\begin{algorithmic}[1]
\If{$u_{t+1} \in U_t$} \Comment{greedily extend the path.}
\State Let $v_{t+1}$ be an arbitrarily chosen endpoint of $P_t$. 
\State Set $V(P_{t+1}) = V(P_t) \cup \{u_{t+1}\}$, $E(P_{t+1}) = E(P_t) \cup \{u_{t+1} v_{t+1} \}$. 
\State Uncolour all of the edges adjacent to $u_{t+1}$.

\ElsIf{$d_{P_t}(u_{t+1}, \scr{L}_t) =1$}     \Comment{path augment via red vertices}
\State Let $x \in \scr{L}_t$ be the (unique) red vertex adjacent to $u_{t+1}$ 

\State Denote $x r \in E(G_t)$ an arbitrary red edge of $x$, and set $v_{t+1} = r$, where $r \in U_t$.
\State Set $V(P_{t+1}) = V(P_t) \cup \{ r \}$ and $E(P_{t+1}) = E(P_t) \cup \{u_{t+1} r ,x r\} \setminus \{u_{t+1}x \}$.
\State Uncolour all of the edges adjacent to $r$.

\ElsIf{$u_{t+1} \in \scr{Q}_t \cup \scr{L}_t$}  \Comment{construct red vertices}
\State Choose $v_{t+1}$ u.a.r.\ from $U_t$.
\State Colour $u_{t+1} v_{t+1}$ red.    \Comment{construct a one-red or two-red vertex}  
\State Set $P_{t+1} = P_t$.

\Else{} \Comment{pass on using edge $u_{t+1} v_{t+1}$.}
\State Choose $v_{t+1}$ arbitrarily from $[n]$.

\State Set $P_{t+1} = P_t$.

\EndIf
\State Update $\scr{Q}_{t+1}$ such that $|\scr{Q}_{t+1}| = |V(P_{t+1})| - 5|\scr{L}_{t+1}|$.

\end{algorithmic}
\end{varalgorithm}

We define the random variables $X(t)= |V(P_t)|$, $L_{1}(t)= |\scr{L}^{1}_t|$, 
$L_{2}(t) = |\scr{L}^{2}_t|$, and $L(t)=|\scr{L}_t|=L_{1}(t)+L_{2}(t)$. Note that $L(t)$
is an auxiliary random variable which we define only for convenience. 
We use $\Delta$ to denote the one step changes in our random variables (i.e., $\Delta X(t): = X(t+1) - X(t)$). Recall that $t_0$ is the step when \ref{alg:fully_randomized} is called.
Let us first show that our random variables cannot change drastically in one round. 

\begin{lemma}[Boundedness Hypothesis -- \ref{alg:fully_randomized}] \label{lem:lipschitz_randomized}
With probability $1 - O(n^{-1})$, 
$$\max\{ |\Delta X(t)|, |\Delta L_{1}(t)|, |\Delta L_{2}(t)|\} = O(\log n) $$
for all $t_0 \le t \le 3n$ with $n -X(t) \ge n/ \log n$. 
\end{lemma}

\begin{proof} 
Note that, by design, the path can only increase its length but it cannot absorb more than one vertex in each round. Hence, the desired property clearly holds for the random variable $X(t)$. To estimate the maximum change for the random variables $L_1(t)$ and $L_2(t)$, we need to upper bound the number of red edges adjacent to any particular unsaturated vertex $v$. Observe that at any step $t \le 3n$, since we have assumed there are at least $n/ \log n$ unsaturated vertices, the number of red edges adjacent to $v$ is stochastically upper bounded by the binomial random variable $\Bin(3n, \log n/n $) with expectation $3 \log n $. It follows immediately from Chernoff's bound that with probability $1-O(n^{-3})$, the number of red edges adjacent to $v$ is $O(\log n)$, and so the desired bound holds by union bounding over all $3n^2$ vertices and steps.
\end{proof}

Let us denote $H_t = ( X(i), L_{1}(i), L_{2}(i))_{0 \le i \le t}$. Note
that $H_t$ does \textit{not} encompass the entire history of the random process
after $t$ rounds (i.e., $G_0, \ldots ,G_t$, the first $t+1$
graphs constructed by the algorithm). This deferred information exposure permits a tractable analysis of the random positioning of $v_t$ when $u_t$ is red.
We observe the following expected difference equations.
\begin{lemma}[Trend Hypothesis -- \ref{alg:fully_randomized}] \label{lem:randomized_expected_differences}
For each $t \ge t_0$, if $n -X(t) \ge n/ \log n$, then
\begin{eqnarray}
\E[\Delta X(t) \mid H_t]&=&  1 - \frac{X(t)}{n} + \frac{2 L(t)}{n} +O(\log n/n) \label{diff-X}\\
\E[ \Delta L_{1}(t)\mid H_t]&=&  \frac{X(t) - 5L(t)}{n} + \frac{2 L_{1}(t)}{n}\left( \frac{2 L_{2}(t)}{n - X(t)} - \frac{L_{1}(t)}{n-X(t)} -1 \right) \nonumber\\
&& + \frac{2 L_{2}(t)}{n}\left( 1 + \frac{2 L_{2}(t)}{n - X(t)} - \frac{L_{1}(t)}{n-X(t)} \right) - \frac{L_{1}(t)}{n} \nonumber\\
&& + \left(1 - \frac{X(t)}{n} \right)\left( \frac{2 L_{2}(t)}{n - X(t)} - \frac{L_{1}(t)}{n - X(t)} \right) + O(\log n/n)\label{diff-L1}\\
\E[ \Delta L_{2}(t) \mid H_t] &=& \frac{L_{1}(t)}{n} - \left(1 - \frac{X(t)}{n}\right) \frac{2 L_{2}(t)}{n-X(t)}  - \frac{2 L_{1}(t)}{n} \frac{2 L_{2}(t)}{n-X(t)}\nonumber\\
&& - \frac{2 L_{2}(t)}{n} \left( 1 + \frac{2 L_{2}(t)}{n - X(t)} \right) + O(\log n/n).\label{diff-L2}
\end{eqnarray}
\end{lemma}

The proof is obtained by examining how the landing of $u_t$ affects the random variables under study. For instance, for $X(t)$, observe that $\Delta X(t)$ is $1$ when $u_{t+1}$ lands on an unsaturated
vertex, or adjacent to a red vertex; and is 0 otherwise. Combining with the probabilities of the above two events yields~(\ref{diff-X}). The proofs for~(\ref{diff-L1}) and~(\ref{diff-L2}) are similar and the details can be found in Appendix~\ref{sec:lemma_proofs}.

In order to analyze \ref{alg:fully_randomized}, we shall employ the differential equation method~\cite{de}. This method is commonly used in probabilistic combinatorics to analyze random processes that evolve step by step. The step changes must be small in relation to the entirety of the discrete structure. For instance, in our application, this refers to adding one edge at a time to the graph on $[n]$ vertices. The method allows us to derive tight bounds on the associated random variables which hold a.a.s. at every step of the random process. We refer the reader
to \cite{bennett_2022} for a gentle introduction to the methodology. The execution of \ref{alg:fully_randomized} starts at some random step $t_0$, which we will prove is a.a.s.\ asymptotic to $s_0n$ for some constant $0 \le s_0 < 1$. Let $X(t_0)$ denote the number of vertices on $P_t$ after the execution of \ref{alg:degree_greedy}. We shall prove that there exists some constant $\hat{x}(s_0)$ such that $|X(t_0)/n - \hat{x}(s_0)| \le \lambda$ for some $\lambda = o(1)$. If $N$ is set to 0, then $t_0=s_0= X(0)= \hat{x}(0)=0$.

Let us now fix a sufficiently small constant $\eps>0$, and define the bounded domain 
$$\scr{D}_{\eps}:= \{ (s,x,\ell_1, \ell_2):  -1 < s < 3, -1 < x < 1 - \eps, |\ell_1| < 2, |\ell_2| <2\}.$$
Consider the system of differential equations in variable $s$
with functions $x =x(s),\ell_1 = \ell_{1}(s)$, and $\ell_{2}=\ell_{2}(s)$:
\begin{eqnarray} 
    x' &=& 1 - x + 2 (\ell_{1} + \ell_2) \label{ode1} \\
    \ell_{1}' &=& x - 5 (\ell_1 + \ell_2) + \ell_1 \left( \frac{2 \ell_2 - \ell_1 }{1-x} -1 \right) 
     + 2\ell_2 \left( 1+  \frac{ 2\ell_2 - \ell_1}{1-x} \right) -\ell_1 + 2 \ell_2 -\ell_1  \label{ode2}  \\
    \ell_2' &=& \ell_1 - 2 \ell_2 - 2\ell_1 \left( \frac{2 \ell_2}{1-x} \right) -2 \ell_2 \left( 1 + \frac{2 \ell_2}{1-x} \right). \label{ode3} 
\end{eqnarray}
The right-hand side (r.h.s.) of each of the above equations is Lipchitz on the domain $\scr{D}_{\eps}$.
Define 
\[
T_{\scr{D}_{\eps}}=\min\{t\ge 0: \ (t/n, X(t)/n,L_1(t)/n,L_2(t)/n)\notin \scr{D}_{\eps}\}.
\]
Now, the `Initial Condition' of Theorem \ref{thm:differential_equation_method} is satisfied
with values $(s_0,\hat{x}(s_0), 0,0)$ and some $\lambda = o(1)$. Moreover,
the `Trend Hypothesis' and `Boundedness Hypothesis' are satisfied with some $\delta=O(\log n/n)$, $\beta=O(\log n)$ and $\gamma=o(n^{-1})$, by Lemmas~\ref{lem:lipschitz_randomized} and~\ref{lem:randomized_expected_differences}. Thus, for every $\delta>0$, $X(t)=nx(t/n)+o(n)$, $L_1(t)=n\ell_1(t/n)+o(n)$ and $L_2(t)=n\ell_2(t/n)+o(n)$ uniformly for all $t_0 \le t \le(\sigma(\eps)-\delta) n$, where $x$, $\ell_1$ and $\ell_2$ are the unique solution to~(\ref{ode1})--(\ref{ode3}) with initial conditions $x(s_0)=\hat{x}(s_0)$, $\ell_1(s_0)=\ell_2(s_0)=0$, and $\sigma(\eps)$ is the supremum of $s$ to which the solution can be extended before reaching the boundary of $\scr{D}_{\eps}$. For $N=0$, $s_0=0$ and the initial conditions are simply $x(0)=\ell_1(0)=\ell_2(0)=0$. This immediately yields the following.

\begin{lemma}[Concentration of \ref{alg:fully_randomized}'s Random Variables] \label{lem1:concentration_random}
For every $\delta>0$, a.a.s.\ for all $t_0 \le t\le (\sigma(\eps)-\delta)n$,
\[
    \max\{|X(t) -x(t/n) n|,|L_{1}(t) - \ell_{1}(t/n) n|, |L_{2}(t) - \ell_{2}(t/n) n| \} =  o(n). 
\]
\end{lemma}

As $\scr{D}_{\eps}\subseteq \scr{D}_{\eps'}$ for every $\eps>\eps'$, $\sigma(\eps)$ is monotonely nondecreasing as $\eps\to 0$. Thus, \begin{equation}
    \alpha^*:=\lim_{\eps \to 0+} \sigma(\eps) \label{alpha_star}
\end{equation} 
exists. It is obvious that $|L_1(t)/n|$ and $|L_2(t)/n|$ are both bounded by 1 for all $t$ and thus, when $t/n$ approaches $\alpha^*$, either $X(t)/n$ approaches 1 or $t/n$ approaches 3. Formally, we have the following proposition.
\begin{proposition}\label{p:boundary} For every $\eps>0$, there exists $\delta>0$ such that a.a.s.\ one of the following holds.
\begin{itemize}
    \item $X(t)>(1-\eps)n$ for all $ t\ge (\alpha^*-\delta)n$;
    \item $\alpha^*=3$.
\end{itemize}

\end{proposition}
The ordinary differential equations~(\ref{ode1})--(\ref{ode3}) do not have an analytical solution. In both cases $N=0$ and $N=100$, numerical solutions show that $\alpha^*<2.1$. (For $N=0$, $\alpha^* \approx 2.07721$.) Thus, by the end of the execution of \ref{alg:fully_randomized}, there are $\eps n$ unabsorbed vertices remaining, for some $\eps=o(1)$.

\subsection{A Clean-up Algorithm} \label{sec:clean_up}

Suppose that we are presented a path $P$ on $(1-\eps)n$ vertices of $[n]$, where $0 < \eps =\eps(n) < 1/1000$.
The assumption on $\eps$ is a mild but convenient assumption. We will apply the argument for $\eps=o(1)$. In this section, we provide an algorithm for the semi-random graph process which absorbs the remaining $\eps n$ vertices into $P$ to form a Hamiltonian path, after which a Hamiltonian cycle can be constructed. The whole procedure takes $O(\sqrt{\eps}n + n^{3/4}\log^2 n) = o(n)$ further steps in the semi-random graph process. Moreover, the algorithm is self-contained in that it only uses the edges of $P$ in its execution.

\begin{lemma}[Clean-up Algorithm] \label{lem:clean_up}
Let $0 < \eps = \eps(n) < 1/1000$, and suppose that $P$ is a path on $(1-\eps)n$ vertices
of $[n]$. Then, given $P$ initially, there exists a strategy for the semi-random graph process which builds
a Hamiltonian cycle from $P$ in $O(\sqrt{\eps}n + n^{3/4}\log^2 n)$ steps a.a.s.
\end{lemma}
\begin{remark}
The constant hidden in the $O(\cdot)$ notation does not depend on $\eps$. The strategy used in the clean-up algorithm is similar to that in~\ref{alg:fully_randomized} but the analysis is done in a much less accurate way, as we only need to prove an $o(n)$ bound on the number of steps required to absorb $\eps n$ vertices. The proof is presented in Appendix~\ref{sec:lemma_proofs}.
\end{remark}

By setting $N=0$ we immediately get an algorithm which a.a.s.\ constructs a Hamiltonian cycle in $\hat\alpha n$ steps, where $\hat\alpha \le 2.07721$. To obtain the better bound in Theorem~\ref{thm:main_upper_bound}, we set $N=100$, and the execution of~\ref{alg:degree_greedy} will be analysed in the next subsection.

\begin{theorem}\label{thm:warm_up}
$\tau_{\mathtt{HAM}} \le \hat{\alpha} \le 2.07721$, where $\hat{\alpha}$ is defined in~(\ref{alpha_star}) with initial conditions for~(\ref{ode1})--(\ref{ode3}) set by $x(0)=\ell_1(0)=\ell_2(0)=0$.
\end{theorem}

\proof This follows by Proposition~\ref{p:boundary}, the numerical value of $\alpha^*$, and Lemma~\ref{lem:clean_up}.\qed

\subsection{A Degree-Greedy Algorithm} \label{sec:degree_greedy}

Let us suppose that after $t \ge 0$ steps, we have constructed the graph $G_t$  which contains the path $P_t$.
As before, our algorithm uses path augmentations, and we colour the edges of $G_t$
to help keep track of when these augmentations can be made. We now use two colours, namely red and blue, to distinguish between edges which are added randomly (red) and greedily (blue). Our blue edges will be chosen so as to minimize the number of blue edges destroyed by path augmentations in future rounds. 

We say that $x \in V(P_t)$ is \textbf{blue}, provided
it is adjacent to a single blue edge of $G_t$, and no red edge. Similarly, $x \in V(P_t)$ is
\textbf{red}, provided it is adjacent to a single red edge of $G_t$, and no blue edge.
Finally, we say that $x \in V(P_t)$ is \textbf{magenta (mixed)}, provided it is adjacent to a single red edge,
and a single blue red. We denote the blue vertices, red vertices, and magenta (mixed) vertices by $\scr{B}_t, \scr{R}_t$
and $\scr{M}_t$, respectively, and define $\scr{L}_t := \scr{B}_t \cup \scr{R}_t \cup \scr{M}_t$ to be
the \textbf{coloured} vertices.
By definition, $\scr{B}_t, \scr{R}_t$ and $\scr{M}_t$ are disjoint. Once again,
we denote our unsaturated vertices by $U_t$, and also maintain a set of
\textbf{permissible} vertices $\scr{Q}_t$ which indicate which saturated vertices
are allowed to be coloured blue. Specifically, using the same reasoning as before,
we ensure the following:
\begin{enumerate}
    \item[(i)] $|\scr{Q}_t| = |V(P_t)| - 5|\scr{L}_t|$. \label{eqn:greedy_permissible_size}
    \item[(ii)] If $\scr{L}_t \neq \emptyset$, then each $x \in \scr{Q}_t$ satisfies $d_{P_t}(x,\scr{L}_t) \ge 3$. \label{eqn:greedy_permissible_distance}
\end{enumerate}


Upon the arrival of $u_{t+1}$, there are five main cases our algorithm must handle. The first
two cases involve extending the length of the path, whereas the latter three describe what to do
when it is not possible to extend the path in the current round.
\begin{enumerate}
    \item If $u_{t+1}$ lands within $U_t$, then greedily extend $P_{t}$. \label{eqn:greedy_extend}
    \item If $u_{t+1}$ lands at path distance one from $x \in \scr{L}_t$, then augment $P_t$ via a coloured edge of $x$,
    where a blue edge is taken over a red edge if possible. \label{eqn:path_augment}
    \item If $u_{t+1}$ lands in $\scr{Q}_t$, then choose $v_{t+1}$ u.a.r.\ amongst those vertices of $U_t$ with \textit{minimum} blue degree. The edge $u_{t+1} v_{t+1}$ is then coloured blue, and a single blue vertex is created. \label{eqn:blue_vertex} 
    \item If $u_{t+1}$ lands in $\scr{R}_t$, then choose $v_{t+1}$ u.a.r.\ amongst those vertices of $U_t$ with minimum blue degree. The edge $u_{t+1} v_{t+1}$ is then coloured blue, and a single red vertex is converted to a magenta (mixed) vertex. \label{eqn:red_vertex} 
    \item If $u_{t+1}$ lands in $\scr{B}_t$, then choose $v_{t+1}$ u.a.r.\ amongst $U_t$ and colour $u_{t+1} v_{t+1}$ red. This case converts a blue vertex to a magenta vertex. \label{eqn:magenta_vertex}
\end{enumerate}
In all the remaining cases, we choose $v_{t+1}$ uniformly at random, and interpret
the algorithm as \textit{passing} on the round. Below we formally describe step $t+1$ of the algorithm upon receiving $u_{t+1}$:
\begin{varalgorithm}{$\mathtt{DegreeGreedy}$}
\caption{Step $t+1$} 
\label{alg:degree_greedy}
\begin{algorithmic}[1]
\If{$u_{t+1} \in U_t$} \Comment{greedily extend the path}
\State Let $v_{t+1}$ be an arbitrarily chosen endpoint of $P_t$. 
\State Set $V(P_{t+1}) = V(P_t) \cup \{u_{t+1}\}$, $E(P_{t+1}) = E(P_t) \cup \{u_{t+1} v_{t+1} \}$. 
\State Uncolour all of the edges adjacent to $u_{t+1}$.  \label{line:unsaturated_uncolour}

\ElsIf{$d(u_{t+1}, \scr{L}_t) =1$}     \Comment{path augment via coloured vertices}
\State Let $x \in \scr{L}_t$ be the (unique) coloured vertex adjacent to $u_{t+1}$

\If{$x$ is red}
\State Denote $x y \in E(G_t)$ the red edge of $x$, where $y \in U_t$. 
\Else{}         \Comment{$x$ is blue or magenta}
\State Denote $x y \in E(G_t)$ the blue edge of $x$, where $y \in U_t$. 
\EndIf

\State Set $v_{t+1} = y$.
\State Set $V(P_{t+1}) = V(P_t) \cup \{ y \}$ and $E(P_{t+1}) = E(P_t) \cup \{u_{t+1} y ,x y\} \setminus \{u_{t+1}x \}$.
\State Uncolour all of the edges adjacent to $y$. \label{line:saturated_uncolour}

\ElsIf{$u_{t+1} \in \scr{Q}_t \cup \scr{R}_t$}   \Comment{construct coloured vertices}
\State Choose $v_{t+1}$ u.a.r.\ from the vertices of $U_t$ of minimum blue degree.
\State Colour $u_{t+1}v_{t+1}$ blue.       \Comment{create a blue or magenta vertex} 
\State Set $P_{t+1} = P_t$. 

\ElsIf{$u_{t+1} \in \scr{B}_t$}
\State Choose $v_{t+1}$ u.a.r.\ from $U_t$.
\State Colour the edge $u_{t+1} v_{t+1}$ red.  \Comment{create a magenta vertex}
\State Set $P_{t+1} = P_t$. 

\Else{} \Comment{pass on using edge $u_{t+1} v_{t+1}$}
\State Choose $v_{t+1}$ u.a.r.\ from $[n]$. 

\State Set $P_{t+1} = P_t$.

\EndIf
\State Update $\scr{Q}_{t+1}$ such that $|\scr{Q}_{t+1}| = |V(P_{t+1})| - 5|\scr{L}_{t+1}|$.    \Comment{update permissible vertices}
\end{algorithmic}
\end{varalgorithm}
\newpage
As in \ref{alg:fully_randomized}, we ensure that all of the algorithm's coloured vertices are at path distance at least $3$ from each other,
and we define a coloured vertex to be \textbf{well spaced} in the same way. Note that red vertices are only created when the blue edges of magenta vertices are uncoloured as a side effect of path extensions and augmentations (see lines \eqref{line:unsaturated_uncolour} and \eqref{line:saturated_uncolour}).

For each $t \ge 0$, define the random variables $X(t) := |V(P_t)|$, $B(t):= |\scr{B}_t|$, 
$R(t) := |\scr{R}_t|$, $M(t) := |\scr{M}_t|$ and $L(t):=|\scr{L}_t|=B(t)+R(t)+M(t)$. 
For each $q \ge 0$ and $t \ge 0$, define $D_{q}(t)$ to be the number of unsaturated vertices adjacent to precisely
$q$ blue edges.  We define the stopping time $\tau_{q}$ to
be the smallest $t \ge 0$ such that $D_{j}(t) = 0$ for all $j<q$, and $D_{q}(t) > 0$. It is obvious that $\tau_q$ is well-defined and is non-decreasing in $q$. By definition, $\tau_{0}=0$. Let us refer to \textbf{phase} $q$ as those $\tau_{q-1} \le t < \tau_{q}$. Observe that during phase $q$, each unsaturated vertex has blue degree $q-1$ or $q$.

\subsection{Analyzing phase $q$}

Suppose that $\tau_{q-1} \le t < \tau_{q}$. It will be convenient to denote $D(t):=D_{q-1}(t)$. 
Given $k_1, k_2 \ge 0$, we say that $y \in U_t$ is of \textbf{type} $(k_1,k_2)$, provided
it is adjacent to $k_1$ blue edges within $\scr{B}_t$ and $k_2$ blue edges within $\scr{M}_t$. Similarly,
$x \in \scr{B}_t \cup \scr{M}_t$ is of type $(k_1,k_2)$, provided its (unique) \textit{blue} edge connects to an unsaturated
vertex of type $(k_1,k_2)$. We denote the number of unsaturated vertices of type $(k_1,k_2)$ by $C_{k_1,k_2}(t)$, the blue vertices of type $(k_1,k_2)$ by $B_{k_1,k_2}(t)$, and the magenta (mixed) vertices of type $(k_1,k_2)$ 
by $M_{k_1,k_2}(t)$. Observe that $B_{k_1,k_2}(t) = k_1 \cdot C_{k_1,k_2}(t)$
and $M_{k_1,k_2}(t) = k_2 \cdot C_{k_1,k_2}(t)$. Moreover,
$D_{j}(t) = \sum_{\substack{k_1, k_2: \\ k_{1} + k_{2} = j}} C_{k_1,k_2}(t)$.

In Appendix \ref{sec:inductive_functions}, we inductively define the functions $x, r$ and
$c_{k_1,k_2}$ for $k_1 + k_2 \ge 0$, as well as a constant $\sigma_q \ge 0$, such that
the following lemma holds:

\begin{lemma}\label{lem:inductive}
A.a.s.\ $\tau_{q} \sim \sigma_{q} n$ for every $0 \le q\le N$.\footnote{For functions $f = f(n)$ and $g=g(n)$, $f \sim g$ is shorthand for $f =(1+ o(1)) g$.} Moreover, at step $\tau_{q}$, a.a.s.
\begin{eqnarray*}
&&X(\tau_{q}) \sim x(\sigma_{q})n, \quad R(\tau_{q}) \sim r(\sigma_{q})n, \\
&&C_{k_1,k_2}(\tau_{q}) \sim c_{k_1,k_2}(\sigma_{q})n \quad \mbox{for all $(k_1,k_2)$ where $k_1+k_2=q$.}
\end{eqnarray*}
\end{lemma}

Although the method in the proof for Lemma~\ref{lem:inductive} is similar to that of Lemmas~\ref{lem:lipschitz_randomized},~\ref{lem:randomized_expected_differences},~\ref{lem1:concentration_random} and Proposition~\ref{p:boundary}, the analysis is much more intricate and involved. We postpone the details to Appendix~\ref{sec:inductive}, and now complete the proof of Theorem~\ref{thm:main_upper_bound}.

\begin{proof}[Proof of Theorem~\ref{thm:main_upper_bound}] Set $N=100$. By Lemma~\ref{lem:inductive}, the execution of \ref{alg:degree_greedy} ends at some step $t_0\sim \sigma_N n$. Moreover, $X(t_0)\sim x(\sigma_N)n$. Numerical computation shows that $\sigma_N\approx 2.00189$. 
Next, the algorithm executes \ref{alg:fully_randomized}. Let $\alpha^*$ be as defined in (\ref{alpha_star}) where the initial conditions to the differential equations (\ref{ode1})--(\ref{ode3}) are set by $s_0=\sigma_N$, $x(s_0)=x(\sigma_N) \approx 0.99991$, 
and $\ell_1(s_0)=\ell_2(s_0)=0$. Numerical computations show that $\alpha^*\approx 2.01678$. By  
Proposition~\ref{p:boundary} and the fact that  $\alpha^*<3$, the execution of the first two stages (\ref{alg:degree_greedy} and \ref{alg:fully_randomized}) finishes at some step $(\alpha^*+o(1))n$, and the number of unsaturated vertices remaining is $o(n)$. Finally, the clean-up algorithm constructs a Hamiltonian cycle with an additional $o(n)$ steps by Lemma~\ref{lem:clean_up}. The theorem follows.
\end{proof}

\section{Proof of Theorem~\ref{thm:main_lower_bound}}
Suppose $G_t$ has a Hamiltonian cycle $H_t = H$ after $t \ge 0$ steps. Recall that
for the (directed) semi-random edge $(u_i,v_i)$, we refer to $u_i$ as its square and $v_i$ as its circle. 
We begin with the following observations:
\begin{enumerate}[label=(\subscript{O}{{\arabic*}})]
\item $H$ uses exactly $n$ squares;     \label{eqn:exact_squares}
    \item $H$ uses at most $2$ squares on each vertex;  \label{eqn:two_squares}
    \item Suppose $(u_i,v_i)$ is an edge of $G_t$, and $v_i$ received at least two squares. Then, either $H$ uses at most one square on $v_i$, or $H$ does not contain the edge $(u_i,v_i)$.  \label{eqn:disallow_edge}
\end{enumerate}
The first two observations above are obvious. For \ref{eqn:disallow_edge}, notice that if $H$ uses exactly $2$ squares on $v_i$, then these $2$ squares correspond to 2 edges in $H$ that are incident to $v_i$. Moreover, neither of these edges can be $(u_i,v_i)$, as $u_i$ is the square of $(u_i,v_i)$. Thus, the edge $(u_i,v_i)$ cannot be used by $H$ as $v_i$ has degree 2 in $H$.

Define $Z_x$ as the number of squares on vertex $x \in [n]$, the observation \ref{eqn:two_squares} above indicates the consideration of the random variable 
\[
Z=\sum_{x=1}^n \left( \bm{1}_{Z_x= 1} + 2\cdot \bm{1}_{Z_x\ge 2}\right)
= 2n - \sum_{x=1}^n \left( 2\cdot \bm{1}_{Z_x= 0} + \bm{1}_{Z_x= 1} \right),
\]
which counts the total number of squares that can possibly contribute to $H$, truncated at 2 for each vertex. Observation \ref{eqn:disallow_edge} above indicates the consideration of the following two sets of structures:

Let $\W_1$ be the set of pairs of vertices $(x,y)$ at time $t$ such that
\begin{enumerate}
    \item[(a)] $x$ receives its first square in some step $i < t$, and $y$ receives the corresponding circle in the same step;
    \item[(b)] no more squares land on $x$ after step $i$;
    \item[(c)] at least two squares land on $y$ after step $i$.
\end{enumerate}

Let $\W_2$ be the set of pairs of vertices $(x,y)$ at time $t$ such that
\begin{enumerate}
    \item[(a)] $x$ receives its first square in some step $i < t$, and $y$ receives the corresponding circle in the same step;
    \item[(b)] exactly one more square lands on $x$ either before or after step $i$;
    \item[(c)] at least two squares land on $y$ after step $s$.
\end{enumerate}

Note that for every $(x,y)\in \W_1$, at most 2 squares on $x$ and $y$ together can be used in $H$, although $x$ and $y$ together contribute $3$ to the value of $Z$. Similarly, for every $(x,y)\in \W_2$, at most 3 squares on $x$ and $y$ together can be used in $H$, although $x$ and $y$ together contribute $4$ to the value of $Z$.

Therefore, the total number of squares contributing to $H$ is at most $Z-|\W_1|-|\W_2|+W$, where $W$ accounts for double counting, which sometimes happens when there are $(x_1,y_1), (x_2,y_2)\in \W_1\cup \W_2$ where $\{x_1,y_1\}\cap \{x_2,y_2\}\neq \emptyset$. 
More precisely, let 
\begin{eqnarray*}
\T_1&=&\{((x_1,y_1),(x_2,y_2))\in \W_1\times \W_2:\  y_1=x_2\}\\
\T_2&=&\{((x_1,y_1),(x_2,y_2))\in \W_2\times \W_2:\ y_1=x_2\}.
\end{eqnarray*}
Then, $W:=|\T_1|+|\T_2|$.\footnote{Note that the cases where $((x_1,y_1),(x_2,y_2))\in \W_1\times \W_1$ such that $y_1=y_2$ and $((x_1,y_1),(x_2,y_2))\in \W_2\times \W_2$ such that  $y_1=x_2$ do not cause double counting.}

The random variable $Z$ is well understood. From the limiting Poisson distribution of the number of squares in a single vertex, we immediately get that a.a.s. $Z\sim (2-2e^{-s}-e^{-s}s)n$ for $s:=t/n$.

We will estimate the expectation of $|\W_1|, |\W_2|, |\T_1|, |\T_2|$ as well as the concentration of these random variables. However, concentration may fail if the semi-random process uses a strategy which places many circles on a single vertex. Intuitively, placing many circles on a single vertex is not a good strategy for quickly building a Hamiltonian cycle, as it wastes many edges. To formalise this idea, let $\mu:=\sqrt{n}$ (indeed, choosing any $\mu$ such that $\mu\to \infty$ and $\mu=o(n)$ will work). We say that a strategy for the semi-random process is $\mu$-\textbf{well-behaved} up until step $t$, if no vertex receives more than $\mu$ circles in the first $t$ steps. In~\cite[Definition 3.2 -- Proposition 3.4]{gao2021perfect}, it was proven that it is sufficient to consider $\mu$-well-behaved strategies in the first $t = O(n)$ steps for establishing a lower bound on the number of steps needed to build a perfect matching. These definitions and proofs can be easily adapted for building Hamilton cycles in an obvious way. We thus omit the details and only give a high-level explanation below.

The key idea is that within $t = O(n)$ steps of any semi-random process, the number of vertices of in-degree greater than $\mu$ is at most $O(n/\mu)=o(n)$. Therefore, if a Hamiltonian cycle $C$ is built in $t$ steps, then the subgraph $H$ of $C$ induced by the set $S$ of vertices of in-degree at most $\mu$ in $G_t$ is a collection of paths spanning all vertices in $S$ which must also contain $n-O(n/\mu)=(1-o(1))n$ edges. We call such a pair $(S,H)$ an \textbf{approximate Hamiltonian cycle}. It follows from the above argument that it takes at least as long time to build a Hamiltonian cycle as to build an approximate Hamiltonian cycle. It is then easy to show by a coupling argument that if a strategy builds an approximate Hamiltonian cycle in $t = O(n)$ steps, then there exists a well-behaved strategy that builds an approximate Hamiltonian cycle in $t$ steps as well. Note that observations \ref{eqn:two_squares}--\ref{eqn:disallow_edge} hold for approximate Hamiltonian cycles, and \ref{eqn:exact_squares} holds for approximate Hamiltonian cycles with $n$ replaced by $(1-o(1))n$. Thus, no approximate Hamiltonian cycles can be built until step $Z-|\W_1|-|\W_2|+W\ge (1-o(1))n$. We now estimate the sizes of $\W_1$, $\W_2$, $\T_1$, and $\T_2$ in the semi-random process when executing a well-behaved strategy $\scr{S}$. Crucially, the sizes of these sets do \textit{not} rely on the decisions made by $\scr{S}$. Recall that $(G^{{\mathcal S}}_{s})_{s\ge 0}$ denotes the sequence of graphs produced by ${\mathcal S}$. 

\begin{lemma}
\label{lem:f}
Suppose ${\mathcal S}$ is $\mu$-well-behaved. For every $t=\Theta(n)$, the following  a.a.s.\ holds in $G^{{\mathcal S}}_{t}$, \[
Z-|\W_1|-|\W_2|+W\sim f(s)n,
\]
where $s:=t/n$ and $f(s)$ is defined in Theorem~\ref{thm:main_lower_bound}.
\end{lemma}

\begin{proof}[Proof of Theorem~\ref{thm:main_lower_bound}]
Recall that $\beta$ is the positive root of
$f(s)=1$. Then, for every $\eps>0$, $Z-|\W_1|-|\W_2|+W\le (1-O(\eps)) n$  a.a.s.\ in $G^{{\mathcal S}}_{(\beta-\eps)n}$ for any $\mu$-well-behaved ${\mathcal S}$. Therefore, $\tau_{\texttt{HAM}}\ge \beta$.
\end{proof}

\section{Conclusion and Open Problems}

We have made significant progress on reducing the gap between the previous best upper and lower bounds on $\tau_{{\tt HAM}}$. 
That being said, we do not believe that any of our new bounds are tight. For instance, in the case of our lower bound, 
one could study the appearance of more complicated substructures which prevent any strategy from building
a Hamiltonian cycle. One way to likely improve the upper bound would be to analyze an adaptive algorithm whose
decisions are all made greedily. Rather, in the terminology of \ref{alg:degree_greedy}, when a (second) square lands
on a blue vertex, the edge is greedily chosen amongst unsaturated vertices of minimum blue degree (opposed to u.a.r.). Unfortunately, it seems challenging to analyze this algorithm via the differential equation method, but it is likely that this algorithm will
lead to an improved upper bound on $\tau_{{\tt HAM}}$ of less than $2$.

Another direction is to understand which graph properties exhibit \textbf{sharp thresholds}.
Given property $\scr{P}$, the definition of $\tau_{\scr{P}}$ ensures that
there exists a strategy $\scr{S}^*$ such that for all $\eps > 0$, $G_{t}^{\scr{S}^*}(n)$ satisfies $\scr{P}$ a.a.s. for $t \ge (\tau_{\scr{P}} + \eps)n$.
On the other hand, $G_{t}^{\scr{S^*}}(n)$ may satisfy $\scr{P}$ with \textit{constant} probability for $t \le (\tau_{\scr{P}} - \eps)n$ without
contradicting the definition of $\tau_{\scr{P}}$. Thus, for $\scr{P}$ to have a sharp threshold, the following property must also hold. For each strategy $\scr{S}$ and $\eps > 0$, if $t \le (\tau_{\scr{P}} - \eps)n$, then a.a.s. $G_{t}^{\scr{S}}(n)$ does \textit{not} satisfy $\scr{P}$. 
It is known that for basic properties, such as minimum degree $k \ge 1$, sharp thresholds do exist \cite{beneliezer2019semirandom}.
Moreover, in \cite{beneliezer2020fast} it was shown that if $H$ is a spanning graph with max degree $\Delta = \omega( \log n)$, then the appearance of $H$ takes $(\Delta/2 + o(\Delta))n$ rounds, and $H$ (deterministically) cannot be constructed in fewer than $\Delta n/2$ rounds.
However, in general it remains open as to whether or not a sharp threshold exists when $H$ is
\textit{sparse} (i.e., $\Delta = O(\log n)$).
Very recently, the second author and Surya \cite{macrury_2022}, a developed
a general machinery for proving the existence of sharp thresholds in adaptive random
graph processes. Applied to the semi-random graph process, they show that
sharp thresholds exist for the property of being Hamiltonian as well as to containing a perfect matching. This provides some evidence that sharp thresholds \textit{do} exist when $\Delta = O(\log n)$, and we leave this an interesting open problem.



\bibliographystyle{plain}

\bibliography{refs.bib}

\appendix

\section{Proofs of Lemmas~\ref{lem:randomized_expected_differences} and~\ref{lem:clean_up}}
\label{sec:lemma_proofs}


\begin{proof}[Proof of Lemma~\ref{lem:randomized_expected_differences}]

As discussed, \ref{alg:fully_randomized} ensures that at time $t$ there are at most $2$ red vertices
which are not well-spaced. Thus, since our expected differences each allow for a $O(\log n/n)$ term, without loss of generality, we can
assume that all our red vertices are well-spaced. Note that all our explanations below assume that we have conditioned on $H_t$.

The first expected difference is easy to see.
Observe that $\Delta X(t)$ is $1$ when $u_{t+1}$ lands on an unsaturated
vertex, or adjacent to a red vertex. Clearly, these are disjoint events, and they occur with probabilities $1 -X(t)/n$ and $2 L(t)/n$, respectively. 

In order to derive the second and third expected differences, 
we make use of the following crucial observation:
\begin{enumerate}[label=(\subscript{O}{{\arabic*}})]
    \item Conditional on $H_t$, the circles of the red edges of $\scr{L}_t$ are distributed u.a.r.\ amongst the unsaturated vertices $U_t$. \label{eqn:conditional}
\end{enumerate}
Note that were we to condition on the full history, i.e., $G_0, \ldots ,G_t$, then these circles would be determined by the history of the process, and so the only randomness in the expectations
would be over the draw of $u_{t+1}$. By averaging over this additional randomness, we are able to get the
claimed expected differences.

Consider now the second expected difference and assume that $u_{t+1}$ lands on
an unsaturated vertex. Firstly, observe that this event
occurs with probability $1 - X(t)/n$. On the other hand, all the red edges adjacent to $u_{t+1}$ 
will be uncoloured after the path augmentation involving $u_{t+1}$ is made. Now, because of \ref{eqn:conditional}, there are $\frac{L_{1}(t)}{n - X(t)}$ red edges belonging to $\scr{L}^{1}_t$ which are adjacent
to $u_{t+1}$ in expectation. After the path augmentation involving $u_{t+1}$, these edges are uncoloured and so $\frac{L_{1}(t)}{n - X(t)}$ one-red vertices are destroyed in expectation. Now, in expectation there are also $\frac{2 L_{2}(t)}{n - X(t)} + O(\log n /n)$ red edges adjacent to $u_{t+1}$ which belong to \textit{distinct} two-red vertices. To see this, fix a two-red vertex $x \in \scr{L}_{t}^{2}$ and observe that because of \ref{eqn:conditional}, precisely one red edge of $x$ is adjacent to $u_{t+1}$ with probability $\frac{2}{n-X(t)} - \frac{1}{(n-X(t))^{2}}$. Since  $n - X(t) \ge n/\log n$ by assumption, this probability is  $\frac{2}{n-X(t)} + O( (\log n/n)^2)$, and so the $\frac{2 L_{2}(t)}{n - X(t)} + O( \log n /n )$ term follows after summing over all the vertices of $\scr{L}_{t}^{2}$. Now, after the path augmentation involving $u_{t+1}$, these red edges are uncoloured. Since these red edges belonged to distinct two-red vertices, the path augmentation creates $\frac{2 L_{2}(t)}{n - X(t)} + O(\log n /n)$ new one-red vertices in expectation. These two cases explain the $\left(1 - \frac{X(t)}{n} \right)\left( \frac{2 L_{2}(t)}{n - X(t)} - \frac{L_{1}(t)}{n - X(t)} \right)$ term.


Let us now consider when $u_{t+1}$ lands on a saturated vertex and $d_{P_t}(u_{t+1}, \scr{L}_t) =1$,
where $x$ is the unique red vertex adjacent to $u_{t+1}$. If $x$ is a one-red vertex, then let
$r$ be such that $xr$ is the red edge of $x$. Observe that after the augmentation, $xr$ will be uncoloured,
and $x$ will no longer be a red vertex. Moreover, in expectation there are $\frac{L_{1}(t)}{n-X(t)} + O(\log n /n)$ other red edges belonging to $\scr{L}^{1}_t$ which will be uncoloured. Thus, $1 + \frac{L_{1}(t)}{n-X(t)} + O( \log n/n)$ one-red vertices will be destroyed in expectation. On the other hand,
there are $\frac{2 L_{2}(t)}{n - X(t)} +O(\log n/n)$ red edges adjacent to $r$ which belong to distinct two-red vertices in expectation. Thus, $\frac{2 L_{2}(t)}{n - X(t)} +O(\log n/n)$ two-red vertices will become one-red vertices in expectation after augmenting via $xr$ and $u_{t+1}r$. Since $u_{t+1}$ lands next
to a one-red vertex with probability, $\frac{2 L_{1}(t)}{n}$, this explains the 
$\frac{2 L_{1}(t)}{n}\left( \frac{2 L_{2}(t)}{n - X(t)} - \frac{L_{1}(t)}{n-X(t)} -1 \right)$ term.
An analogous argument explains the $\frac{2 L_{2}(t)}{n}\left( 1 + \frac{2 L_{2}(t)}{n - X(t)} - \frac{L_{1}(t)}{n-X(t)} \right)$ term.

Consider when $u_{t+1}$ lands in $\scr{Q}_t$. Observe that this occurs
with probability 
$\frac{|\scr{Q}_t|}{n} = \frac{X(t) - 5L(t)}{n} $. In this case,
$v_{t+1}$ is chosen u.a.r.\ amongst $U_t$ and $u_{t+1} v_{t+1}$
is coloured red. Thus, $u_{t+1}$ becomes a red vertex,
and so  $L_{1}(t)$ increases by $1$ and we get $\Delta L_{1}(t) =1$. This explains the 
$\frac{X(t) - 5L(t)}{n}$ term.

The final case to consider is when $u_{t+1}$ lands on a saturated vertex,
and $u_{t+1} \in \scr{L}^{1}_t$. Observe
that this occurs with probability $\frac{L_{1}(t)}{n}$. Moreover,
the algorithm will then choose $v_{t+1}$ u.a.r.\ amongst $U_t$ and colour the edge $u_{t+1}v_{t+1}$
red. After this move, $u_{t+1}$ will be converted from a one-red vertex to a two-red vertex, and so
$\Delta L_{1}(t) =1$. This explains the $\frac{-L_{1}(t)}{n}$ term.

By combining the contributions from all of the above cases, we get the second expected difference.
The third expected difference follows via an analogous argument.
\end{proof}

\begin{proof} [Proof of Lemma~\ref{lem:clean_up}]
Let $j_0=\eps n$.
For each $k\ge 1$, 
let $j_k=(1/2) j_{k-1}$ if $j_{k-1}>n^{1/4}$, and let $j_k=j_{k-1}-1$ otherwise. Clearly, $j_k$ is a decreasing function of $k$. Let $\tau_1$ be the smallest natural number $k$ such that $j_k\le n^{1/4}$. Let $\tau$ be the natural number $k$ such that $j_k=0$. Obviously, $\tau_1=O(\log n)$ and $\tau=O(n^{1/4})$.

We use a cleaning-up algorithm, which runs in iterations. The $k$-th iteration repeatedly absorbs $j_{k-1}-j_k$ vertices into $P$, leaving $j_k$ unsaturated vertices in the end. The $k$-th iteration of the cleaning-up algorithm works as follows.

\begin{itemize}
    \item[(i)] ({\em Initialising}): Uncolour all vertices in the graph;
    \item[(ii)] ({\em Building reservoir}): Let $m_k:=\sqrt{\eps}(1/2)^{k/2}n$
     for $k\le \tau_1$ and $m_k:=n^{1/2}$ if $\tau_1<k\le\tau$. 
    Add $m_k$ semi-random edges as follows. If $u_t$ lands on an unsaturated vertex, a red vertex or a neighbour of a red vertex in $P$, then let $v_t$ be chosen arbitrarily. This edge $u_tv_t$ will not be used in our construction. Otherwise, colour $u_t$ red and choose an arbitrary $v_t$ among those unsaturated vertices with the minimum  number of red neighbours. Colour $u_tv_t$ red. Note that each red vertex is adjacent to exactly one red edge; 
    \item[(iii)] ({\em Absorbing via path augmentations}): Add semi-random edges as follows. Suppose that $u_t$ lands on $P$ and at least one neighbour of $u_t$ on $P$ is red. (Otherwise, $v_t$ is chosen arbitrarily, and this edge will not be used in our construction.) Let $x$ be such red vertex (if $u_t$ has two neighbours on $P$ that are red, then select one of them arbitrarily). Let $y$ by the neighbour of $x$ where $xy$ is red, and let $v_t=y$. Extend $P$ by deleting the edge $xu_t$ and adding the edges $xy$ and $yu_t$. Uncolour all red edges incident to $y$ and all red neighbours of $y$ (which, of course, includes vertex $x$). 
\end{itemize}
Notice that, in each iteration, $m_k\ge n^{1/2}$. 
Indeed, this is obviously true for $\tau_1<k\le\tau$. On the other hand, if $k\le \tau_1$, then $j_k = \eps n (1/2)^k$ and so $m_k = \sqrt{n j_k} \ge \sqrt{n}$ (in fact, $m_k = \Omega(n^{5/8})$).

Let $T_k$ denote the length of the $k$-th iteration of the cleaning-up algorithm.
It remains to prove that a.a.s.\ $\sum_{k\le \tau} T_k = O(\sqrt{\eps} n + n^{3/4}\log^2 n)$.
Let $R_k$ be the number of red vertices obtained after step (ii) of iteration $k$. Obviously, $R_k\le m_k$.  On the other hand, each $u_t$ is coloured red with probability at least $1-j_{k-1}/n-3m_k/n \ge 1 - \eps - 3\sqrt{\eps} \ge 0.95$. Hence, $R_k$ can be stochastically lower bounded by the binomial random variable $\Bin(m_k, 0.95)$. By the Chernoff bound, with probability at least $1-n^{-1}$, $R_k \ge 0.9 m_k$, as $m_k\ge n^{1/2}$. 

First, we consider iterations $k\le \tau_1$.
Let $\tilde R_k$ be the number of red vertices at the end of step~(iii). Note that the minimum degree property of step~(ii) ensures each unsaturated vertex is adjacent to at most $R_k/j_{k-1}+1 \le m_k/j_{k-1}+1$ red vertices. Moreover, exactly $j_{k-1}-j_k=(1/2)j_{k-1}$ vertices are absorbed in step (iii). As a result, 
\[
\tilde R_k\ge R_k-\left(\frac{m_k}{j_{k-1}}+1\right) \cdot \frac{j_{k-1}}{2} \ge 0.9 m_k-\frac{m_k}{2}-\frac{j_{k-1}}{2} \ge 0.3m_k,
\]
as $j_{k-1} = 2 j_k \le 2 \sqrt{\eps} m_k \le 0.1 m_k$.
It follows that throughout step (iii), there are at least $0.3 m_k$ red vertices. Thus, for each semi-random edge added to the graph, the probability that a path extension can be performed is at least $0.3m_k/n=0.3\sqrt{\eps}(1/2)^{k/2}$. Again, by the Chernoff bound, with probability at least $1-n^{-1}$, the number of semi-random edges added in step (iii) is at most
\[
2(j_{k-1}-j_k) \cdot \frac{2^{k/2}}{0.3\sqrt{\eps}} \le 7\sqrt{\eps}(1/2)^{k/2} n.
\]
Combining the number of semi-random edges added in step (ii), it follows that with probability at least $1-n^{-1}$,
$T_k\le m_k+7\sqrt{\eps}(1/2)^{k/2} n = 8\sqrt{\eps}(1/2)^{k/2} n$.

Next, consider iterations $\tau_1<k\le \tau$. In each iteration, exactly one unsaturated vertex gets absorbed. The number of semi-random edges added in step (ii) is $m_k=n^{1/2}$. We have argued that with probability at least $1-n^{-1}$, $R_k\ge 0.9m_k$.
Thus, for each semi-random edge added to the graph, the probability that a path extension can be performed is at least $0.9m_k/n=0.9n^{-1/2}$. By the Chernoff bound, with probability at least $1-n^{-1}$, the number of semi-random edges added in step (iii) is at most
$
n^{1/2}\log^2 n.
$ Thus, with probability at least $1-n^{-1}$, $T_k\le n^{1/2}+n^{1/2}\log^2 n \le 2n^{1/2}\log^2 n$.

Taking the union bound over all $k\le\tau$, since $\tau = O(n^{1/4})$, it follows that a.a.s.\ 
\[
\sum_{k\le \tau}T_k \le \sum_{k\le \tau_1} 8\sqrt{\eps}(1/2)^{k/2} n + \sum_{\tau_1<k\le \tau} 2n^{1/2}\log^2 n 
=O(\sqrt{\eps}n + n^{3/4}\log^2 n) \]

We have shown that a.a.s.\ by adding $O(\sqrt{\eps}n + n^{3/4}\log^2 n)$ additional semi-random edges we can construct a Hamiltonian path $P$. To complete the job and turn it into a Hamiltonian cycle, let $u$ and $v$ denote the left and, respectively, the right endpoint of $P$. First, add $n^{1/2}$ semi-random edges $u_tv_t$ where $v_t$ is always $u$, discarding any multiple edges that could possibly be created. For each such semi-random edge $u_tu$, colour the left neighbour of $u_t$ on $P$ blue. Next, add add $n^{1/2}\log^2 n$ semi-random edges $u_tv_t$ where $v_t$ is always $v$. Suppose that some $u_t=x$ is blue. Then, a Hamiltonian cycle is obtained by deleting $xy$ from $P$ and adding the edges $xv$ and $uy$, where $y$ is the right neighbour of $x$ on $P$. By Chernoff bound, a.a.s.\ a semi-random edge added during the second round hits a blue vertex, completing the proof.
\end{proof}

\section{Proof of Lemma~\ref{lem:inductive}}
\label{sec:inductive}

We once again must first argue that our random variables cannot change drastically in one round during phase $q$.
\begin{lemma}[Lipschitz Condition -- \ref{alg:degree_greedy}] \label{lem:boundedness_greedy}
If $|\Delta C(t)| := \max_{\substack{k_1, k_2 \in \Nn \cup \{0\}: \\ k_1 + k_2 \in \{q-1,q\}}} |\Delta C_{k_1,k_2}(t)|$, then with probability $1- O(n^{-1})$,
$$\max\{ |\Delta X(t)|, |\Delta C(t)|, |\Delta R(t)|\} = O(\log n)$$
for all $\tau_{q-1} \le t < \tau_{q}$ with $n - X(t) = \Omega(n)$. 
\end{lemma}
\begin{proof}
Since $q \le N$ is a constant which does not depend on $n$, we can apply the same argument to bound
the red edges of each $\Delta C_{k_1,k_2}(t)$ as in Lemma \ref{lem:lipschitz_randomized}, and then union bound over all $k_1,k_2 \ge 0$ such that $k_1 + k_2 \in \{q-1,q\}$.
\end{proof}

Let $H_t$ denote the history of the above random variables during the first $t$ rounds.
We now state the conditional expected differences of our random variables,
where we assume that $\tau_{q-1} \le t < \tau_{q}$ is such that $n - X(t) = \Omega(n)$.
Firstly, observe that once again:
\begin{eqnarray}\label{eqn:degree_x_change}
\E[\Delta X(t)\mid H_t]&=&  1 - \frac{X(t)}{n} + \frac{2 L(t)}{n} +O(1/n)
\end{eqnarray}
We now consider $\Delta R(t)$:
\begin{eqnarray} \label{eqn:degree_r_change}
\E[\Delta R(t) \mid H_t] &=& \frac{M(t)}{n}  - \frac{R(t)}{n} -  \frac{2(B(t) + M(t))}{n} \frac{R(t)}{(n-X(t))}  \nonumber \\
&&  + \sum_{\substack{j,h: \\ j+ h \in \{q-1,q\}}}  \frac{2 h ( M_{j,h}(t) +  B_{j,h}(t) ) }{n} \nonumber \\
&&- \frac{2 R(t)}{n} \left( 1 + \frac{R(t)}{n-X(t)}  - \frac{M(t)}{n- X(t)} \right) - \frac{R(t)}{n} +O(1/n) 
\end{eqnarray}
Consider $\Delta C_{k_1,k_2}(t)$ and first assume that $k_1 + k_2 = q-1$:
\begin{eqnarray} \label{eqn:degree_c_min_change}
\E[ \Delta C_{k_1,k_2}(t) \mid H_t] &=& 
\frac{M_{k_1 -1,k_2 +1}(t)}{n} \cdot \bm{1}_{k_1 >0} -\frac{C_{k_1,k_2}(t)}{n} - \frac{M_{k_1,k_2}(t)}{n} \nonumber \\
&&+\frac{2 (B(t) + M(t))}{n} \left( \frac{M_{k_1 -1, k_2 +1}(t)}{n - X(t)} \cdot \bm{1}_{k_1 > 0} - \frac{M_{k_1,k_2}(t)}{n - X(t)} \right) - \frac{2 (B_{k_1,k_2}(t) + M_{k_1,k_2}(t))}{n} \nonumber \\
&& +\frac{2 R(t)}{n } \left( \frac{M_{k_1 -1,k_2 +1}(t)}{n - X(t)} \cdot \bm{1}_{k_1 >0} - \frac{M_{k_1,k_2}(t)}{n - X(t)} - \frac{C_{k_1,k_2}(t)}{n - X(t)} \right) \nonumber \\
&&-\frac{(X(t)- 5 L(t))}{n} \frac{C_{k_1,k_2}(t)}{D(t)}  \nonumber \\
&&+ \frac{B_{k_1 +1,k_2 -1}(t)}{n} \cdot \bm{1}_{k_2 > 0} -\frac{R(t)}{n} \frac{C_{k_1,k_2}(t)}{D(t)} -\frac{B_{k_1,k_2}(t)}{n} + O(1/n)
\end{eqnarray}

When $k_1 + k_2 =q$, two terms from the above expression are modified slightly, and have their signs
reversed:
\begin{eqnarray} \label{eqn:degree_c_max_change}
\E[ \Delta C_{k_1,k_2}(t) \mid H_t] &=& 
\frac{M_{k_1 -1,k_2 +1}(t)}{n} \cdot \bm{1}_{k_1 >0} -\frac{C_{k_1,k_2}(t)}{n} - \frac{M_{k_1,k_2}(t)}{n} \nonumber \\
&&+\frac{2 (B(t) + M(t))}{n} \left( \frac{M_{k_1 -1, k_2 +1}(t)}{n - X(t)} \cdot \bm{1}_{k_1 > 0} - \frac{M_{k_1,k_2}(t)}{n - X(t)} \right) - \frac{2 (B_{k_1,k_2}(t) + M_{k_1,k_2}(t))}{n} \nonumber \\
&& +\frac{2 R(t)}{n } \left( \frac{M_{k_1 -1,k_2 +1}(t)}{n - X(t)} \cdot \bm{1}_{k_1 >0} - \frac{M_{k_1,k_2}(t)}{n - X(t)} - \frac{C_{k_1,k_2}(t)}{n - X(t)} \right)   \nonumber \\
&&+\frac{(X(t)- 5 L(t))}{n} \frac{C_{k_1-1,k_2}(t)}{D(t)}  \nonumber  \\
&&+ \frac{B_{k_1 +1,k_2 -1}(t)}{n} \cdot \bm{1}_{k_2 > 0} +\frac{R(t)}{n} \frac{C_{k_1,k_2-1}(t)}{D(t)} -\frac{B_{k_1,k_2}(t)}{n} + O(1/n)
\end{eqnarray}



\subsection{Proving the Expected Differences}
Consider step $t$ of phase $q$ which we assume satisfies $n- X(t) = \Omega(n)$.
Similarly to \ref{alg:fully_randomized},~\ref{alg:degree_greedy} ensures that all but at most
$2$ of its coloured vertices are well-spaced. Thus, we may assume without loss of generality that all of our coloured vertices are well-spaced. Since every well-spaced coloured vertex has its own two neighbouring vertices where a path augmentation may occur, it follows that
$\E[\Delta X(t)\mid H_t]= 1 - \frac{X(t)}{n} + \frac{2 L(t)}{n} +O(1/n)$.

In order to prove the remaining expected differences, we analyze the expected values of the random
variables $\Delta R(t)$ and $\Delta C_{k_1,k_2}(t)$ when $u_{t+1}$ lands in subset $A \subseteq [n]$
for a number of choices of $A$. More precisely, we derive tables for $\mb{E}[ \Delta R(t) \cdot \bm{1}_{u_{t+1} \in A} \mid H_t]$ 
and $\mb{E}[ \Delta C_{k_1,k_2}(t) \cdot \bm{1}_{u_{t+1} \in A} \mid H_t]$
when $A \subseteq [n]$ varies across a number of subsets. Since these
are disjoint subsets of $[n]$, and the random variables are $0$ if $u_{t+1}$ lands outside
of these subsets, we can sum the second column entries to get the claimed expected differences.
Note that the entries of our tables do \textit{not} contain the often necessary
$O(1/n)$ term. 

In the below derivations, we once again make use of the following crucial observation:
\begin{enumerate}[label=(\subscript{O}{{\arabic*}})]
    \item Conditional on $H_t$, the circles of the red edges of $\scr{L}_t$ are distributed u.a.r.\ amongst the unsaturated vertices $U_t$. \label{eqn:conditional_degree_greedy}
\end{enumerate}
In all our below explanations, we assume that we have conditioned on $H_t$. We also abuse
notation and simultaneously identify our random variables as sets (i.e., $C_{k_1,k_2}(t)$ denotes the set of
unsaturated vertices of type $(k_1,k_2)$ after $t$ steps).
\begin{table}[H]
\caption{Expected Changes to $\Delta R(t)$}\label{table:red_changes}
\begin{tabular}{ |c|c|}
 \hline
 $A \subseteq [n]$  & $\mb{E}[ \Delta R(t) \cdot \bm{1}_{u_{t+1} \in A} \mid H_t]$\\
 \hline
 $U_t$ & $\frac{M(t)}{n} -\frac{R(t)}{n}$ \\
 Path distance $1$ from $\scr{B}_t \cup \scr{M}_t$  & $-\frac{2 (B(t) + M(t) )}{n} \frac{R(t)}{n - X(t)} + \sum_{\substack{j,h: \\ j+ h \in \{q-1,q\}}} \frac{2( B_{j,h}(t)  + M_{j,h}(t)) \cdot h}{n} $  \\ 
 Path distance $1$ from $\scr{R}_t$  & $-\frac{2 R(t)}{n} \left(1 + \frac{R(t)}{n - X(t)} \right) + \frac{2 R(t)}{n} \frac{M(t)}{n-X(t)}$\\
  $\scr{R}_t$   & $\frac{-R(t)}{n}$ \\
 \hline
\end{tabular}
\end{table}
\begin{proof}[Proof of Table \ref{table:red_changes}]
We provide complete proofs only of row entries $1$ and $3$, as $2$ follows similarly to
$3$ and $4$ follows similarly to $1$.

In order to see the first row entry, observe that there are $\frac{R(t)}{n}$ red edges between $\scr{R}_t$
and $U_t$. Moreover, if $u_{t+1}$ lands on an unsaturated vertex, then a path extension is made, and all the
edges adjacent to $u_{t+1}$ are uncoloured. Thus, in expectation $\frac{R(t)}{n}$ red vertices are destroyed
when $u_{t+1}$ lands in $U_t$. We also claim that $\frac{M(t)}{n}$ red vertices are created in expectation
when $u_{t+1}$ lands in $U_t$. In order to see this, suppose that $x$ is a magenta vertex, with blue edge $xb$ and red edge $xr$,
where $b, r \in U_t$. Now, if $r \neq b$ and  $u_{t+1}$ lands on $b$, then the edge $xb$ is uncoloured, and
$x$ becomes a red vertex. On the other hand, $r$ is distributed u.a.r. amongst $U_t$ (see \ref{eqn:conditional_degree_greedy}), and in particular, is distinct
from $b$ with probability $1 - \frac{1}{n-X(t)}$. Thus, $x$ is converted to a red vertex with probability $\frac{1}{n} \left(1 - \frac{1}{n-X(t)} \right) = 1/n + O(1/n^2)$, since $n - X(t) = \Omega(n)$.
By summing over all the magenta vertices, we recover a $M(t)/n + O(1/n)$ term which implies the $M(t)/n - R(t)/n$ row entry.

Let us now consider what happens when $u_{t+1}$ lands on $P_t$ next to a red vertex. Clearly, this event occurs with probability $2 R(t)/n$, as we have assumed that all the red vertices are well-spaced. Let us suppose that $u_{t+1}$ lands on red vertex $x$, and its red edge is $xr$ for some $r \in U_t$. In this case, a path augmentation involving $xr$ occurs and all the adjacent edges of $r$ are uncoloured after $r$ is added to the path. By \ref{eqn:conditional_degree_greedy}, there are $1 + \frac{R(t) -1}{n-X(t)}$ red edges adjacent to $r$ in expectation. This accounts for the $-\frac{2 R(t)}{n} \left(1 + \frac{R(t)}{n - X(t)} \right)$ term. In order to see the second term of this row entry, first observe that 
\begin{equation}\label{eqn:magenta_vertex_identity}
M(t) =\sum_{\substack{j,h: \\ j+ h \in \{q-1,q\}}}  h \cdot C_{j,h}(t).
\end{equation}
Now, fix $j, h \ge 0$ such that $j + h \in \{q-1,q\}$, and consider $c \in C_{j,h}(t)$. Observe that by definition, $c$ is adjacent to precisely $h$ magenta vertices of $M_{j,h}(t)$ via blue edges. 
We claim that all $h$ of these magenta vertices will be reclassified as red vertices provided
the following events occur:
\begin{itemize}
    \item The circle $v_{t+1}$ lands on $c$.
    \item All the red edges of these $h$ magenta vertices are \textit{not} adjacent to $c$.
\end{itemize}
By \ref{eqn:conditional_degree_greedy}, the first event occurs with probability $1/(n-X(t))$. 
and the latter with probability $\left(1 - \frac{1}{n - X(t)} \right)^h$. Thus, since $h$ is a constant, each $c  \in C_{j,h}(t)$
contributes 
\[ \frac{h}{n-X(t)} \left(1 - \frac{1}{n - X(t)} \right)^h = \frac{h}{n-X(t)} + O(1/n^2) \]
red vertices in expectation.
By summing over $c \in C_{j,h}(t)$, this accounts for the $\frac{h \cdot C_{j,h}(t)}{n-X(t)}$ term.
Finally, we sum over $j,h$ and apply \eqref{eqn:magenta_vertex_identity} to get the $\frac{M(t)}{n - X(t)}$ term.

 
\end{proof}


Consider now $\Delta C_{k_1,k_2}(t)$, where $k_{1} + k_{2} = q-1$. 
\begin{table}[H]
\caption{Expected Changes to $\Delta C_{k_1,k_2}(t)$ for $k_{1} + k_{2} = q-1$} \label{table:type_changes_min}
\begin{tabular}{ |c|c|} 
 \hline
 $A \subseteq [n]$  & $\mb{E}[ \Delta C_{k_1,k_2}(t) \cdot \bm{1}_{u_{t+1} \in A} \mid H_t]$ \\
 \hline
 $U_t$ &    $ \frac{M_{k_1 -1,k_2 +1}(t)}{n} \cdot \bm{1}_{k_1 >0} -\frac{C_{k_1,k_2}(t)}{n} - \frac{M_{k_1,k_2}(t)}{n}$  \\
 Path distance $1$ from $\scr{B}_t \cup \scr{M}_t$  &  $\frac{2 (B(t) + M(t))}{n} \left( \frac{M_{k_1 -1,k_2 +1}(t)}{n - X(t)} \cdot \bm{1}_{k_1 > 0} - \frac{M_{k_1,k_2}(t)}{n - X(t)} \right) - \frac{2 (B_{k_1,k_2}(t) + M_{k_1,k_2}(t))}{n}    $ \\ 
 Path distance $1$ from $\scr{R}_t$  &  $\frac{2 R(t)}{n } \left( \frac{M_{k_1 -1,k_2 +1}(t)}{n - X(t)} \cdot \bm{1}_{k_1 >0} - \frac{M_{k_1,k_2}(t)}{n - X(t)} - \frac{C_{k_1,k_2}(t)}{n - X(t)} \right)$ \\
$\scr{Q}_t$ &  $-\frac{(X(t)- 5 L(t))}{n} \frac{C_{k_1,k_2}(t)}{D(t)}$ \\
  $\scr{R}_t$   &  $-\frac{R(t)}{n} \frac{C_{k_1,k_2}(t)}{D(t)}$ \\
   $\scr{B}_t$   &   $-\frac{B_{k_1,k_2}(t)}{n} + \frac{B_{k_1 +1,k_2 -1}(t)}{n} \cdot \bm{1}_{k_2 > 0}$       \\
 \hline
\end{tabular}
\end{table}
\begin{proof}[Proof of Table \ref{table:type_changes_min}]
Assume that $k_1, k_2$ are both non-zero, as this is the most involved case. We provide complete 
proofs of row entries $1,3$ and $4$. Row entry $2$ follows similarly to row entry $3$, and row entry $5$ follows
similarly to row entry $4$. Row entry $5$ has a simple proof, so we omit it.

We begin with row entry $1$ when $u_{t+1}$ lands on an unsaturated vertex.
First consider the $-\left(\frac{C_{k_1,k_2}(t)}{n} + \frac{M_{k_1,k_2}(t)}{n}\right)$ term. Fix a unsaturated vertex $c$ of type $(k_1,k_2)$, and suppose that $x_1, \ldots , x_{k_2}$ are its magenta neighbours of type $(k_1,k_2 )$. By definition, $c x_i$ is coloured blue, and each $x_i$ also has a red edge $r_i x_i$ for $i=1,\ldots, k_2$. We claim that $c$ is destroyed
with probability $(k_2 + 1)/n + O(1/n^2)$. In order to see this, first observe that because of \ref{eqn:conditional_degree_greedy},
the vertices $c, r_1, \ldots ,r_{k_2}$ are distinct with probability
\[
    \prod_{i=1}^{k_2}\left(1 - \frac{i}{n-X(t)}\right) = 1 + O(1/n).
\]
Let us now condition on these vertices being distinct. Observe
that if $u_{t+1}$ lands on $c$, then $u_{t+1}$ is added to the path, and so $c$ is destroyed.
On the other hand, if $u_{t+1}$ lands on $r_i$, then $r_i$ is added to the path, and so the edge $r_i x_i$ is no longer red.
In particular, $x_i$ is converted to a blue vertex, and so the type of $c$ is reclassified as $(k_1 +1, k_2 -1)$.
In either case, $c$ is destroyed. Thus, conditional on the vertices $b, r_{1}, \ldots ,r_{k_2}$ being distinct,
$c$ is destroyed with probability $(k_2 + 1)/n$. As such, $c$ is destroyed with the claimed probability of $(k_2 + 1)/n + O(1/n^2)$.
By summing over all $c \in C_{k_1,k_2}(t)$ and using the fact that $k_2 \cdot C_{k_1,k_2}(t) = M_{k_1,k_2}(t)$, we attain the $-\left(\frac{C_{k_1,k_2}(t)}{n} + \frac{M_{k_1,k_2}(t)}{n}\right)$ term.
The  $\frac{M_{k_1 -1,k_2 +1}(t)}{n}$ term follows similarly, where reclassifying unsaturated vertices of type $(k_1 -1, k_2 +1)$ to type $(k_1, k_2)$ causes $C_{k_1,k_2}(t)$ to increase.


We now consider the third row entry when $u_{t+1}$ lands at path distance one from $\scr{R}_t$.
First observe that this event occurs with probability $2 R(t)/n$. At this point, using the previous argument and 
\ref{eqn:conditional_degree_greedy}, the path augmentation involving $v_{t+1}$ causes each $c \in C_{k_1,k_2}(t)$
to be destroyed with probability $(k_2 + 1)/(n - X(t)) + O(1/n^2)$. Moreover, each $c' \in C_{k_1 -1,k_2 +1}(t)$
is reclassified as type $(k_1,k_2)$ with probability $(k_2 + 1)/(n - X(t)) + O(1/n^2)$. By summing over
$c \in C_{k_1,k_2}(t)$ and $c' \in C_{k_1 -1,k_2 +1}(t)$, we attain the $\frac{M_{k_1 -1,k_2 +1}(t)}{n - X(t)} - \frac{M_{k_1,k_2}(t)}{n - X(t)} - \frac{C_{k_1,k_2}(t)}{n - X(t)}$ expression. After multiplying
by $2 R(t)/n$, we attain row entry $3$.

Let us now consider row entry $4$ when $u_{t+1}$ lands on a permissible vertex $x \in \scr{Q}_t$.
Clearly, this event occurs with probability $|\scr{Q}_t|/n = (X(t) - 5L(t))/n$. In this case, the algorithm chooses $v_{t+1}$ u.a.r. amongst $D(t)$,
the unsaturated vertices of minimum degree $q-1$, and colours the edge $x v_{t+1}$ blue. Thus, if we fix $c \in C_{k_1,k_2}(t)$, then $c$ will be chosen
with probability $1/D(t)$ since $k_1 + k_2 = q-1$. In this case, $c$ gains a blue edge connected to a blue vertex,
and thus will be reclassified as type $(k_1 + 1, k_2)$. Thus, each $c \in C_{k_1,k_2}(t)$ will be reclassified with probability 
\[
\frac{X(t) - 5L(t)}{n} \frac{1}{D(t)}.
\]
By summing over all $c \in C_{k_1,k_2}(t)$, we get the $-\frac{(X(t)- 5 L(t))}{n} \frac{C_{k_1,k_2}(t)}{D(t)}$ term.
\end{proof}
Finally, when $k_{1} + k_{2} = q$, the expressions in rows $(4)$ and $(5)$ are modified slightly.
\begin{table}[H]
\caption{Expected Changes to $\Delta C_{k_1,k_2}(t)$ for $k_{1} + k_{2} = q$} \label{table:type_changes_max}
\begin{tabular}{ |c|c|} 
 \hline
 $A \subseteq [n]$  & $\mb{E}[ \Delta C_{k_1,k_2}(t) \cdot \bm{1}_{u_{t+1} \in A} \mid H_t]$ \\
 \hline
 $U_t$ &    $ \frac{M_{k_1 -1,k_2 +1}(t)}{n} \cdot \bm{1}_{k_1 >0} -\frac{C_{k_1,k_2}(t)}{n} - \frac{M_{k_1,k_2}(t)}{n}$  \\
 Path distance $1$ from $\scr{B}_t \cup \scr{M}_t$  &  $\frac{2 (B(t) + M(t))}{n} \left( \frac{M_{k_1 -1,k_2 +1}(t)}{n - X(t)} \cdot \bm{1}_{k_1 > 0} - \frac{M_{k_1,k_2}(t)}{n - X(t)} \right) - \frac{2 (B_{k_1,k_2}(t) + M_{k_1,k_2}(t))}{n}    $ \\ 
 
 Path distance $1$ from $\scr{R}_t$  &  $\frac{2 R(t)}{n } \left( \frac{M_{k_1 -1,k_2 +1}(t)}{n - X(t)} \cdot \bm{1}_{k_1 >0} - \frac{M_{k_1,k_2}(t)}{n - X(t)} - \frac{C_{k_1,k_2}(t)}{n - X(t)} \right)$ \\
 $\scr{Q}_t$ &  $\frac{(X(t)- 5 L(t))}{n} \frac{C_{k_1 -1,k_2}(t)}{D(t)}$ \\
  $\scr{R}_t$   &  $\frac{R(t)}{n} \frac{C_{k_1,k_2-1}(t)}{D(t)}$ \\
   $\scr{B}_t$   &   $-\frac{B_{k_1,k_2}(t)}{n} + \frac{B_{k_1 +1,k_2 -1}(t)}{n} \cdot \bm{1}_{k_2 > 0}$       \\
 \hline
\end{tabular}
\end{table}
\begin{proof}[Proof of Table \ref{table:type_changes_max}]
The explanations for the case of $k_1 + k_2 = q$ are identical to those of $k_1 + k_2 = q-1$, except
that vertices of type $(k_1,k_2)$ are created (instead of destroyed) when $u_{t+1}$ satisfies
$u_{t+1} \in \scr{Q}_t$ or $u_{t+1} \in \scr{R}_t$.

\end{proof}

\subsection{Proving Lemma \ref{lem:inductive}} \label{sec:inductive_functions}
In this section, we inductively prove Lemma~\ref{lem:inductive}. Firstly, when $q=0$, by definition $\tau_{0} =0$, and so $\sigma_{0} := 0$ trivially satisfies the conditions of Lemma \ref{lem:inductive}. Let us now assume that $q\ge 1$ and for each of $0\le i\le q-1$ we have defined $\sigma_i$ and functions $x,r$ and $c_{j,h}$ on $[0, \sigma_{i}]$ for each $j,h \ge 0$ with  $j + h = i$, and Lemma~\ref{lem:inductive} holds for all $0\le i\le q-1$. We shall define $\sigma_q$ which satisfies  $\sigma_{q} > \sigma_{q-1}$, extend each $x, r$ and $c_{j,h}$ to $[0, \sigma_q]$, and define new functions $c_{k_1,k_2}$ on $[0,\sigma_q]$ for $k_1 + k_2 =q$. We shall then prove that these functions satisfy the assertion of Lemma~\ref{lem:inductive} with respect to $\tau_q$ and $\sigma_q$, which will complete the proof of the lemma.

Fix a sufficiently small constant $\eps > 0$, and define the bounded domain 
\[
\scr{D}_{\eps}:= \left\{(s,x,r, (c_{j,h})_{j + h \in \{q-1,q\}}): \sigma_{q-1} - 1 <s < 3, |x| < 1 - \eps, |r| <2, |c_{j,h}| <2, \eps < \sum_{j,h: \, j+ h =q-1} c_{j,h} < 2\right\}.
\]
It will be convenient to define auxiliary functions to simplify our equations below.
Specifically, set $b_{k_1,k_2} = k_1 \cdot c_{k_1,k_2}$ and $m_{k_1,k_2} := k_2 \cdot c_{k_1,k_2}$,
as well as $b = \sum_{\substack{j,h: \\ j+ h \in \{q-1,q\} }} b_{j,h}$ and $m = \sum_{\substack{j,h: \\ j+ h \in \{q-1,q\}}} m_{j,h}$.
Finally, set $d = \sum_{\substack{j,h: \\ j+ h =q-1}} c_{j,h}$.
Observe the following system of differential equations:
\begin{eqnarray} \label{eqn:greedy_x_and_r}
x' &=& 1 - x + 2 \\
r' &=& m  - r -  \frac{2(b + m)r}{1-x} + \sum_{\substack{j,h: \\ j+ h \in \{q-1,q\}}} 2 h (m_{j,h} +  b_{j,h}) \nonumber \\
    &&  - 2 r \left( 1 + \frac{r}{1-x}  - \frac{m}{1-x} \right) - r   
\end{eqnarray}
If $k_1 + k_2 = q-1$, then:
\begin{eqnarray}\label{eqn:greedy_de_min_type}
c_{k_1,k_2}' &=& m_{k_1 -1,k_2 +1} \cdot \bm{1}_{k_1 >0} - c_{k_1,k_2} - m_{k_1,k_2} \nonumber \\ 
&&+ 2 (b + m ) \left( \frac{m_{k_1 -1, k_2 +1}\cdot \bm{1}_{k_1 > 0} - m_{k_1,k_2}}{1 - x} \right) - 2 (b_{k_1,k_2} + m_{k_1,k_2} )\nonumber \\
&& + 2r \left( \frac{m_{k_1 -1,k_2 +1}\cdot \bm{1}_{k_1 >0} - m_{k_1,k_2}  - c_{k_1,k_2}}{1-x} \right)   \nonumber  \\
&&- (x- 5 \ell) \frac{c_{k_1,k_2}}{d}  + b_{k_1 +1,k_2 -1} \cdot \bm{1}_{k_2 > 0} - r \frac{c_{k_1,k_2}}{d} - b_{k_1,k_2}  
\end{eqnarray}
Otherwise, if $k_1 + k_2 =q$, then:
\begin{eqnarray}\label{eqn:greedy_de_max_type}
c_{k_1,k_2}' &=& m_{k_1 -1,k_2 +1} \cdot \bm{1}_{k_1 >0} - c_{k_1,k_2} - m_{k_1,k_2} \nonumber \\
&&+ 2 (b + m ) \left( \frac{m_{k_1 -1, k_2 +1}\cdot \bm{1}_{k_1 > 0} - m_{k_1,k_2}}{1 - x} \right) - 2 (b_{k_1,k_2} + m_{k_1,k_2}  )\nonumber \\
&& + 2r \left( \frac{m_{k_1 -1,k_2 +1}\cdot \bm{1}_{k_1 >0} - m_{k_1,k_2}  - c_{k_1,k_2}}{1-x} \right)   \nonumber  \\
&& + (x- 5 \ell) \frac{c_{k_{1} -1,k_2}}{d}  + b_{k_1 +1,k_2 -1} \cdot \bm{1}_{k_2 > 0} + r \frac{c_{k_1,k_2 -1}}{d} - b_{k_1,k_2}
\end{eqnarray}
The right-hand side (r.h.s.) of each of the above equations is Lipchitz on the domain $\scr{D}_{\eps}$, as $d$ is bounded
below by $\eps$. Define
\[
T_{\scr{D}_{\eps}}:=\min\{t \ge 0: (t/n, X(t)/n, R(t)/n, (C_{k_1,k_2}(t)/n)_{k_1 + k_2 \in \{q,q-1\}}) \notin \scr{D}_{\eps}\}
\]
By the inductive assumption, the `Initial Condition' of Theorem \ref{thm:differential_equation_method} is satisfied for
some $\lambda = o(1)$ and values $\sigma_{q-1}, x(\sigma_{q-1}), r(\sigma_{q-1})$ and $c_{j,h}(\sigma_{q-1})$,
where $c_{j,h}(\sigma_{q-1}):=0$ for $j +h = q$. Moreover, the `Trend Hypothesis'
is satisfied with $\delta = O(1/n)$, by the expected differences of \eqref{eqn:degree_x_change}-\eqref{eqn:degree_c_max_change}.
Finally, the `Boundedness Hypothesis' is satisfied with $\beta=O(\log n)$ and $\gamma=O(n^{-1})$
by Lemma~\ref{lem:boundedness_greedy}. Thus, by Theorem~\ref{thm:differential_equation_method}, for every $\delta>0$, a.a.s.\ $X(t)=nx(t/n)+o(n)$, $R(t)=n r(t/n)+o(n)$ and $C_{k_1,k_2}(t)=n c_{k_1,k_2}(t/n)+o(n)$ uniformly for all $\sigma_{q-1} n \le t \le (\sigma(\eps)-\delta) n$, where $x$, $\ell_1$ and $c_{k_1,k_2}$ are the unique solution to \eqref{eqn:greedy_x_and_r}-\eqref{eqn:greedy_de_max_type} with the above initial conditions, and $\sigma(\eps)$ is the supremum of $s$ to which the solution can be extended before reaching the boundary of $\scr{D}_{\eps}$. This immediately yields the following lemma.

\begin{lemma}[Concentration of \ref{alg:degree_greedy}'s Random Variables] \label{lem:concentration_random}
For every $\delta > 0$, a.a.s. for all  $\tau_{q-1} \le t \le (\sigma(\eps)-\delta) n$
and $k_1, k_2 \ge 0$ such that $k_1 + k_2 \in \{q,q-1\}$,
\[
    \max\{|X(t) -x(t/n) n|,|R(t) - r(t/n) n|, |C_{k_1,k_2}(t) - c_{k_1,k_2}(t/n) n||\} =  o(n). 
\]
\end{lemma}
As $\scr{D}_{\eps}\subseteq \scr{D}_{\eps'}$ for every $\eps>\eps'$, $\sigma(\eps)$ is monotonely nondecreasing as $\eps\to 0$, and so $\sigma_{q}:=\lim_{\eps\to 0+}\sigma(\eps)$ exists. Moreover, the derivatives of the functions $x, r,$ and $c_{k_1,k_2}$
are uniformly bounded on $(\sigma_{q-1}, \sigma_q)$, which implies that the functions must be uniformly continuous. The latter condition implies that
the functions are (uniquely) continuously extendable to $[\sigma_{q-1}, \sigma_q]$, and so the following limits exist:
\begin{eqnarray}
 x(\sigma_{q})&:=&\lim_{s\to \sigma_{q}-} x(s) \label{eqn:x_limit} \\ 
 r(\sigma_{q})&:=&\lim_{s\to \sigma_{q}-} r(s)  \label{eqn:r_limit} \\ 
 c_{k_1,k_2}(\sigma_{q})&:=&\lim_{s\to \sigma_{q}-} c_{k_1,k_2}(s). \label{eqn:c_limit}
\end{eqnarray}
Random variables $|R(t)/n|$ and $|C_{k_1,k_2}(t)/n|$ for $k_1 + k_2 \in \{q,q-1\}$ are both bounded by $1$ for all $t$.
Thus, when $t/n$ approaches $\sigma_q$, $X(t)/n$ approaches $1$, or $t/n$ approaches $3$, or $D(t)/n:= \sum_{\substack{j,h :\\ j + h =q-1}} C_{j,h}(t)/n$ approaches $0$. Formally, we have the following proposition:
\begin{proposition}\label{prop:degree_boundary} For every $\eps>0$, there exists $\delta>0$ such that a.a.s.\ one of the following holds.
\begin{itemize}
    \item $D(t) <  \eps n$ for all $ t\ge (\sigma_q-\delta)n$;
    \item $X(t)>(1-\eps)n$ for all $ t\ge (\sigma_q-\delta)n$;
    \item $\sigma_q=3$.
\end{itemize}
\end{proposition}
The ordinary differential equations \eqref{eqn:greedy_x_and_r}-\eqref{eqn:greedy_de_max_type} again do not have an analytical solution. However, numerical solutions show that $\sigma_q < 3$, and $x(\sigma_q) < 1$.  Thus, after executing \ref{alg:degree_greedy}
for $t = \sigma_{q} n + o(n)$ steps, there are $D(t) < \eps n$ vertices of type $q-1$ remaining for some $\eps = o(1)$.
At this point, by observing the numerical solution (\ref{eqn:x_limit})--(\ref{eqn:c_limit}) at $\sigma_q$, we know that there exists some absolute constant $0 < p < 1$ such that $(X(t) - 5 L(t))/n \ge p$,
where we recall that $L(t)$ counts the total number of coloured vertices at time $t$. Hence, at each step, some vertex of type $q-1$ becomes of type $q$ with probability at least $p$. Thus, by applying Chernoff's bound,
one can show that a.a.s.\ after another $O(\eps n/p)=o(n)$ rounds, all vertices of type $q-1$ are destroyed. It follows that a.a.s.\ $|\tau_{q}/n - \sigma_q| =o(1)$, and so Lemma \ref{lem:inductive} is proven.
\section{Proof of Lemma~\ref{lem:f}}
\begin{lemma} \label{lem:structures} Suppose ${\mathcal S}$ is $\mu$-well-behaved up until time $t = \Theta(n)$. A.a.s.\ the following holds for $G^{{\mathcal S}}_{t}$: 
\begin{align}
 |\W_1|&\sim n \sum_{i\le t} \frac{1}{n} \left(1-\frac{1}{n}\right)^{t} \sum_{i\le j_1<j_2\le t} \frac{1}{n^2}\left(1-\frac{1}{n}\right)^{j_2}\label{W1} \displaybreak[0] \\
|\W_2|&\sim n \sum_{i_1\le i_2\le t} \frac{1}{n^2}\left(1-\frac{1}{n}\right)^{t} \left(\sum_{i_1\le j_1<j_2\le t} \frac{1}{n^2}\left(1-\frac{1}{n}\right)^{j_2} + \sum_{i_2<j_1<j_2\le t} \frac{1}{n^2}\left(1-\frac{1}{n}\right)^{j_2}\right) \label{W2} \\
 |\T_1|&\sim n \sum_{i\le t} \frac{1}{n} \left(1-\frac{1}{n}\right)^{t} \sum_{i\le j_1<j_2\le t} \frac{1}{n^2}\left(1-\frac{1}{n}\right)^{t}\nonumber\\
 &\times\left(\sum_{j_1\le h_1<h_2\le t} \frac{1}{n^2}\left(1-\frac{1}{n}\right)^{h_2} + \sum_{j_2<h_1<h_2\le t} \frac{1}{n^2}\left(1-\frac{1}{n}\right)^{h_2}\right)\label{T1} \\
 |\T_2|&\sim n \sum_{i_1<i_2\le t} \frac{1}{n^2} \left(1-\frac{1}{n}\right)^{t} \sum_{i_1\le j_1<j_2\le t} \frac{1}{n^2}\left(1-\frac{1}{n}\right)^{t}\nonumber\\
 &\times\left(\sum_{j_1\le h_1<h_2\le t} \frac{1}{n^2}\left(1-\frac{1}{n}\right)^{h_2} + \sum_{j_2<h_1<h_2\le t} \frac{1}{n^2}\left(1-\frac{1}{n}\right)^{h_2}\right)\nonumber\\
 & + \sum_{i_1<i_2\le t} \frac{1}{n^2} \left(1-\frac{1}{n}\right)^{t} \sum_{i_2\le j_1<j_2\le t} \frac{1}{n^2}\left(1-\frac{1}{n}\right)^{t}\nonumber\\
 &\times\left(\sum_{j_1\le h_1<h_2\le t} \frac{1}{n^2}\left(1-\frac{1}{n}\right)^{h_2} + \sum_{j_2<h_1<h_2\le t} \frac{1}{n^2}\left(1-\frac{1}{n}\right)^{h_2}\right).\label{T2}
\end{align}
\end{lemma}

\begin{proof} 
We prove~(\ref{W1}) and briefly explain the expressions in~(\ref{W2})--(\ref{T2}) whose proofs are similar to that of~(\ref{W1}). Fix a vertex $x \in[n]$ and a square $u_i$ for $i\le t$. The probability that $u_i$ lands on $x$ in step $i$ is $1/n$. Condition on this event. The probability that $x$ receives no squares in any steps other than $i$ is $(1-1/n)^{t-1}\sim (1-1/n)^{t}$. Let $y$ be the vertex which the strategy chooses to pair with $u_i$ with. Fix any two integers $i<j_1<j_2\le t$, the probability that $y$ receives its
first two squares at times $j_1$ and $j_2$ is $n^{-2}(1-1/n)^{j_2-2}\sim n^{-2}(1-1/n)^{j_2}$. Summing over all possible values of $i,j_1,j_2$ and multiplying by $n$, the number of choices for $x$, gives $\ex |\W_1|$.

For concentration of $|\W_1|$ we prove that $\ex |\W_1|^2\sim (\ex |\W_1|)^2$. For any pair of $((x_1,y_1),(x_2,y_2))$ in $ \W_1\times \W_1$, either $x_1,y_1,x_2,y_2$ are pairwise distinct, or $y_1=y_2$. It is easy to see that the expected number of pairs where $x_1,y_1,x_2,y_2$ are pairwise distinct is 
\[
n^2 \sum_{\substack{i_1\le t\\ i_2\le t}} \frac{1}{n^2} \left(1-\frac{1}{n}\right)^{2(t-1)} \sum_{\substack{i_1\le j_1<j_2\le t\\
i_2\le h_1<h_2\le t
}} \frac{1}{n^4}\left(1-\frac{1}{n}\right)^{j_2-2+h_2-2} \sim (\ex|\W_1|)^2.
\]
The expected number of pairs where $y_1=y_2$ is at most $\mu n$ as there are most $n$ choices for $x_1$ and given $(x_1,y_1)$, there can be at most $\mu$ choices for $(x_2,y_2)$ since ${\mathcal S}$ is $\mu$-well-behaved. Since $\mu=o(n)$, $\mu n=o(n^2)$ which is $o((\ex|\W_1|)^2)$. Thus we have verified that $\ex |\W_1|^2\sim (\ex|\W_1|)^2$ and thus by the second moment method, a.a.s.\ $|\W_1|\sim \ex|\W_1|$. 

The proofs for the expectation and concentration of $|\W_2|$, $|\T_1|$ and $|\T_2|$ are similar. We briefly explain the expressions in~(\ref{W2})--(\ref{T2}):

In~(\ref{W2}), $i_1$ and $i_2$ denote the two steps at which $x$ receives a square. Since there are two squares on $x$, there are two choices of circles, namely $v_{i_1}$ and $v_{i_2}$. The two summations over $(j_1,j_2)$ accounts for the two choices of $v_{i_1}$ and $v_{i_2}$, depending on which
is to be covered by two squares. Thus, $j_1$ and $j_2$ denote the steps where the first two squares on $v_{i_1}$ or $v_{i_2}$ arrive. 

In~(\ref{T1}), $i$ denotes the step where $x_1$ receives its only square; $j_1$ and $j_2$ denote the two steps where $y_1=x_2$ receives its two squares. Hence, there are two choices for $y_2$, and $h_1$ and $h_2$ denote the two steps of the first two squares $y_2$ receives.

In~(\ref{T2}), $i_1$ and $i_2$ denote the two steps where $x_1$ receives its two squares---hence there are two choices for $y_1$. Integers $j_1$ and $j_2$ denote the two steps where $y_1=x_1$ receives its two squares---hence there are two choices for $y_2$. Finally, $h_1$ and $h_2$ denote the steps where $y_2$ receives its first two squares. 
\end{proof}

From Lemma~\ref{lem:structures}, we deduce that for $t = sn$,
\begin{align*}
\displaybreak[0]
|\W_1|&\sim n e^{-s}\int_{0}^s dx\int_{x}^s dy_1 \int_{y_1}^s e^{-y_2} dy_2 = n e^{-s}\left(1-\frac{e^{-s}s^2}{2}-e^{-s}s-e^{-s}\right) \\
|\W_2|&\sim  ne^{-s}\int_{0}^s d x_1 \int_{x_1}^s d x_2 \left(\int_{x_1}^s d y_1 \int_{y_1}^s e^{-y_2} d y_2 + \int_{x_2}^s d y_1 \int_{y_1}^s e^{-y_2} d y_2  \right)\\
&= n e^{-s} \left(s-e^{-s}s^2 - \frac{e^{-s}s^3}{2} - e^{-s}s\right) \\
|\T_1|&\sim n e^{-2s} \int_{0}^s dx \int_{x}^s dy_1 \int_{y_1}^s d y_2 \left(\int_{y_1}^s d z_1 \int_{z_1}^s e^{-z_2} d z_2 + \int_{y_2}^s d z_1 \int_{z_1}^s e^{-z_2} d z_2 \right) \\
&= n e^{-2s}\left(-1+s - \frac{e^{-s}s^3}{3}-\frac{e^{-s}s^2}{2}-\frac{e^{-s}s^4}{8}+e^s\right)\\
|\T_2|&\sim n e^{-2s} \int_{0}^s dx_1 \int_{x_1}^s dx_2 \int_{x_1}^s dy_1 \int_{y_1}^s d y_2 \left(\int_{y_1}^s d z_1 \int_{z_1}^s e^{-z_2} d z_2 + \int_{y_2}^s d z_1 \int_{z_1}^t e^{-z_2} d z_2 \right)\\
&+n e^{-2s} \int_{0}^s dx_1 \int_{x_1}^s dx_2 \int_{x_2}^s dy_1 \int_{y_1}^s d y_2 \left(\int_{y_1}^s d z_1 \int_{z_1}^s e^{-z_2} d z_2 + \int_{y_2}^s d z_1 \int_{z_1}^s e^{-z_2} d z_2 \right)\\
&= n e^{-2s}\left(-s+s^2-e^{-s}s\left(\frac{s^4}{8}+\frac{s^3}{3}+\frac{s^2}{2}-1\right)\right)
\end{align*}
It follows now that
$Z-|\W_1|-|\W_2|+W \sim f(s) n$ where recall that
\[
f(s)=2+e^{-3s}(s+1)\left(1-\frac{s^2}{2}-\frac{s^3}{3}-\frac{s^4}{8}\right)+e^{-2s}\left(2s+\frac{5s^2}{2}+\frac{s^3}{2}\right)-e^{-s}\left(3+2s\right).
\]

\section{The Differential Equation Method}
In this section, we provide a self-contained \textit{non-asymptotic} statement of the differential equation method. 
The statement combines \cite[Theorem $2$]{warnke2019wormalds}, and its extension \cite[Lemma $9$]{warnke2019wormalds}, in a form convenient for our purposes, where we modify the notation of \cite{warnke2019wormalds} slightly. In particular, we rewrite \cite[Lemma $9$]{warnke2019wormalds} in
a less general form in terms of a stopping time $T$. We need only check the `Boundedness Hypothesis' (see below) for $0 \le t \le T$,
which is exactly the setting of Lemmas \ref{lem:lipschitz_randomized} and \ref{lem:boundedness_greedy}.

Suppose we are given integers $a,n \ge 1$, a bounded domain $\scr{D} \subseteq \mb{R}^{a+1}$,
and functions $(F_k)_{1 \le k \le a}$ where each $F_k: \scr{D} \to \mb{R}$ is $L$-Lipschitz-continuous on $\scr{D}$
for $L \ge 0$. Moreover, suppose that $R \in [1, \infty)$ and $S \in (0, \infty)$ are \textit{any} constants which satisfy $\max_{1 \le k \le a} |F_{k}(x)| \le R$ for all $x=(s,y_1,\ldots ,y_{a})\in \scr{D}$ and $0 \le s \le S$.
\begin{theorem}[Differential Equation Method, \cite{warnke2019wormalds}] \label{thm:differential_equation_method}
Suppose we are given $\sigma$-fields $\scr{F}_{0}  \subseteq \scr{F}_{1} \subseteq \cdots$,
and for each $t \ge 0$, random variables $((Y_{k}(t))_{1 \le k \le a}$ which are $\scr{F}_t$-measurable. Define $T_{\scr{D}}$ to be the minimum $t \ge 0$ such that
\[
 (t/n, Y_{1}(t)/n, \ldots , Y_{k}(t)/n) \notin \scr{D}.
\]
Let $T \ge 0$ be an (arbitrary) stopping time\footnote{The stopping time $T\ge 0$ is \textbf{adapted} to $(\scr{F}_t)_{t \ge 0}$, provided
the event $\{\tau = t\}$ is $\scr{F}_t$-measurable for each $t \ge 0$.} adapted to $(\scr{F}_t)_{t \ge 0}$, and
assume that the following conditions hold for $\delta, \beta, \gamma \ge 0$ and
$\lambda \ge \delta \min\{S, L^{-1}\} + R/n$:
\begin{enumerate}
    \item[(i)] The `Initial Condition': For some $(0,\hat{y}_1,\ldots ,\hat{y}_a) \in \scr{D}$, \label{enum:initial_conditions}
    \[
    \max_{1 \le k \le a} |Y_{k}(0) - \hat{y}_k n| \le \lambda n.
    \] 
    \item[(ii)] The `Trend Hypothesis': For each  $t \le \min\{ T, T_{\scr{D}} -1\}$, \label{enum:trend_hypothesis}
    $$|\mb{E}[ Y_{k}(t+1) - Y_{k}(t) \mid \scr{F}_t] - F_{k}(t/n,Y_{1}(t)/n,\ldots ,Y_{a}(t)/n)| \le \delta.$$
    \item[(iii)] The `Boundedness Hypothesis': With probability $1 - \gamma$, \label{enum:boundedness_hypothesis}
    $$|Y_{k}(t+1) -  Y_{k}(t)| \le \beta,$$
    for each $t \le \min\{ T, T_{\scr{D}} -1\}$:
\end{enumerate}
Then, with probability at least $1 - 2a \exp\left(\frac{-n \lambda^2}{8 S \beta^2}\right) - \gamma$, we have that
\begin{equation}
    \max_{0 \le t \le \min\{T, \sigma n\}} \max_{1 \le k \le a} |Y_{k}(t) -y_{k}(t/n) n| < 3 \lambda \exp(L S)n,
\end{equation}
where $(y_{k}(s))_{1 \le k \le a}$ is the unique solution to the system of differential equations
\begin{equation} \label{eqn:general_de_system}
    y_{k}'(s) = F_{k}(s, y_{1}(s),\ldots ,y_{a}(s)) \quad \mbox{with $y_{k}(0) = \hat{y}_k$ for $1 \le k \le a$,}
\end{equation}
and $\sigma = \sigma(\hat{y}_1,\ldots ,\hat{y}_a) \in [0,S]$ is any choice of $\sigma \ge 0$ with the property that
$(s,y_{1}(s),\ldots, y_{a}(s))$ has $\ell^{\infty}$-distance at least $3 \lambda \exp(LS)$ from the boundary of $\scr{D}$
for all $s \in [0, \sigma)$.
\begin{remark}
Standard results for differential equations guarantee that \eqref{eqn:general_de_system} has a unique solution $(y_{k}(s))_{1\le k \le a}$ which extends arbitrarily close to the boundary of $\scr{D}$. 
\end{remark}


\end{theorem}

\end{document}